\documentclass[11pt,leqno]{article}

\usepackage[all]{xy}
\usepackage{amssymb,amsmath,amsthm,    url,rotating}
\usepackage{mathrsfs}
  \usepackage{graphicx,epsfig}
  \usepackage{xcolor,graphicx}

    \usepackage{makeidx}
\newcommand{\n}{\noindent}

\newcommand{\vp}{\varepsilon}
\newcommand{\bb}[1]{\mathbb{#1}}
\newcommand{\cl}[1]{\mathcal{#1}}

\newcommand{\ovl}{\overline}

\theoremstyle{plain}
\newtheorem{thm}{Theorem}[section]
\newtheorem{lem}[thm]{Lemma}

\newtheorem{pro}[thm]{Proposition}

\newtheorem{cor}[thm]{Corollary}

\theoremstyle{definition}

\newtheorem{dfn}[thm]{Definition}

\theoremstyle{remark}
\newtheorem{rem}[thm]{Remark}

\numberwithin{equation}{section}

\def\tilde{\widetilde}

\renewcommand{\tilde}{\widetilde}

\def\C{\bb  C}
\def\CC{\bb  C}

\def\F{\bb  F}

\def\d{\delta}
\def\NN{\bb  N}

\def\CC{\bb  C}

\def\CC{\bb  C}

\def\F{\bb  F}

\def\d{\delta}
\def\NN{\bb  N}

\def\CC{\bb  C}

\def\ov{\overline}
\def\phi{\varphi}

\def\n{\noindent}
\def\nl{\nolimits}

\newtheorem{ass1}{Assertion}
\newtheorem{ass}{}

\begin{document}

\title{A  non-nuclear  $C^*$-algebra  with  the  Weak  Expectation  Property      and  the  Local  Lifting  Property}

\author{by\\
Gilles  Pisier\footnote{ORCID    0000-0002-3091-2049}   \\
Texas  A\&M  University\\
College  Station,  TX  77843,  U.  S.  A.}

\def\C{\mathscr{C}}
\def\B{\mathscr{B}}
\def\I{\cl  I}
\def\e{\cl  E}
 
  \def\a{\alpha}
 
  \maketitle
\begin{abstract}    
We  construct  the first example of a $C^*$-algebra  $A$  with the properties in the title.
This  gives  a  new  example  of  non-nuclear  $A$  for  which  there  is
a  unique  $C^*$-norm  on  $A  \otimes  A^{op}$.
This  example  is  of  particular  interest  in  connection  with the  Connes-Kirchberg  problem,    which  is  equivalent  to  the  question  whether  $C^*({\bb  F}_2)$,  which  is  known  to  have  the  LLP, also  has  the  WEP. Our  $C^*$-algebra  $A$  has  the  same  collection  of finite  dimensional  operator  subspaces  as  $C^*({\bb  F}_2)$ or $C^*({\bb  F}_\infty)$. In addition our example can be made to be quasidiagonal
and of similarity degree (or length) 3.
 In the second part of the paper we reformulate our construction in the more general
framework of a $C^*$-algebra that 
  can be described as the \emph{limit both inductive and projective}
for a sequence of $C^*$-algebras $(C_n)$ 
when each $C_n$ is a \emph{subquotient} of $C_{n+1}$.
We use this to show that for certain local properties of
 injective (non-surjective) $*$-homomorphisms,
there are  $C^*$-algebras for which   the identity map has the same   properties as the $*$-homomorphisms.   \end{abstract}
  
  \medskip{MSC (2010): 46L06, 46L07, 46L09} 
  
    \medskip
  
 The concept of  nuclearity has had a major impact on operator algebra theory (see e.g. \cite{Bla,LauR,RSt,ER}). It was introduced
 by Takesaki in  \cite{[Ta2]} but the name itself was coined by Lance \cite{Lan} in analogy with Grothendieck's  nuclear locally convex spaces.
  A $C^*$-algebra  $A$ is called nuclear if there is a unique $C^*$-norm
  on the algebraic tensor product $A\otimes B$ for \emph{any} other
  $C^*$-algebra  $B$.  By  classical results
   due to Takesaki and Guichardet there is a minimal  and a maximal $C^*$-norm
  on $A\otimes B$ (see e.g. \cite{Tak}), so that the nuclearity of $A$ can be written simply as the  identity   
  of the respective completions that is: $A\otimes_{\min} B=A\otimes_{\max} B$. In his seminal paper \cite{Lan}, Lance asked whether for the latter to hold for all $B$ it suffices to have it for  $B=A^{op}$ (the opposite $C^*$-algebra, i.e.  the same one but with product in reversed order). This question
  was motivated by the case of von Neumann algebras
  for which the pair $A \otimes A^{op}$  admits a fundamental faithful
  representation that later came to be called ``the standard form" (see \cite{Tak2}).
 In  \cite{Kir}  Kirchberg  gave a negative answer (see Remark \ref{ap1})
    by constructing  the  first  example  of  a  non-nuclear  $C^*$-algebra  $A$
such  that  
\begin{equation}  \label{e13}  A  \otimes_{\min}  A^{op}=A  \otimes_{\max}  A^{op}.\end{equation}
In  the  first  lines  of  that  paper  \cite{Kir},    he  observed  that  this  could  be  viewed
as  the  analogue  for  $C^*$-algebras  of  the  author's  result  in  \cite{P0}
for  Banach  space  tensor  products.  It  was  thus  tempting
to  try  to  adapt  the  Banach  space  approach  in  \cite{P0}  to  the    $C^*$-algebra  setting  to  produce
new  examples  satisfying  \eqref{e13}.
In  some  sense  the  present  paper  is  the  result  of  this  quest  but  it      started  to  
be  more  than  wishful  thinking    only    recently.

Kirchberg  (\cite{Kiuhf}  see  also  \cite{P})  proved  that  if  $A$  has  the    Weak  Expectation  Property  (WEP in short, defined in \S \ref{wep})  and  $B$  the  Local  Lifting  Property  (LLP in short, defined in \S \ref{llp})  then
  \begin{equation}  \label{e14}  A  \otimes_{\min}  B=A  \otimes_{\max}  B.\end{equation}
  Thus  if  a  $C^*$-algebra  $A$  has  both  WEP  and  LLP,
  then  \eqref{e14}  holds  with  $A=B$  and  in  fact  since  both  LLP  and  WEP
  remain  valid  for  $A^{op}$  we  have  \eqref{e13}.
  
  Kirchberg  also  proved  (\cite{Kiuhf})  that    $C^*(\F_\infty)$  has  the  LLP.  This  is  in  some  sense  the  prototypical  example  of  LLP,  just  like  $B(H)$  is  for  the  WEP.
  
  The    WEP,  originally  introduced  by  Lance
  \cite{Lan},  has  drawn  more  attention  recently  because
  of  Kirchberg's  work  \cite{Kir}
  and  in  particular  his  proof  that  
  the  Connes  embedding  problem
  is  equivalent  to  the  assertion  that  $C^*(\F_\infty)$  (or  $C^*(\F_2)$)
  satisfies  \eqref{e13}
  or  equivalently  that  it    has  the  WEP.
 We will refer to this as the  Connes-Kirchberg  problem.
 
Note that any $A$ satisfying \eqref{e13} must have the WEP.
In fact the WEP of $A$ is equivalent to   \eqref{e13} 
restricted to ``positive definite tensors", see Remark \ref{rus1}.

The  main  result  of  this  paper  is  the  construction  of  a  non-nuclear  (and  even  non  exact)
separable  $C^*$-algebra  $A$  with  both  WEP  and  LLP.  This  answers  a  question,  that  although
 it remained  implicit    in  Kirchberg's  work (but it is explicitly
in \cite[p. 383]{BO}),  was  clearly  in  the  back  of  his  mind  when  he  produced  the  $A$  satisfying  \eqref{e13}.  
But  since  at  the  time  he  conjectured  the  equivalence  of  WEP  and  LLP,
 the  question  did  not  seem  so  natural  until  the  latter  equivalence  was  disproved  by Junge and the author in  \cite{JP}.
   More specifically, Kirchberg conjectured that  $C^*(\F_\infty)$ and $B(H)$ both satisfy \eqref{e13}
    but the latter was disproved  in  \cite{JP}, while the former is the still open   
     Connes-Kirchberg  problem.

While  we  cannot  prove    \eqref{e13}  for  $C^*(\F_\infty)$,  our  algebra  $A$  has  the  same  collection  of
finite  dimensional  operator  subspaces  as  $C^*(\F_\infty)$.  Thus    our  construction  might
shed  some  light,  one  way  or  the  other,  on  the  Connes-Kirchberg  embedding  problem.

In the second part, we prove a generalization,
which can be viewed as a sort of inductive limit construction for a sequence
of $C^*$-algebras $(C_n)$ where $C_n$ is for each $n$ a subquotient
of $C_{n+1}$. When the $C_n$'s all have the LLP
this construction produces a $C^*$-algebra $A$ for which 
the identity map on $A$ possesses some of the 
tensor product properties of the linking maps. 
For instance, if  the linking map is the embedding $C^*(\F_\infty) \to B(H)$   (and   $B(H)$ is viewed
as a quotient of $C^*(\F) $ for some free group $\F$),  which in some sense has both WEP and LLP, 
we recover our main example.

\medskip

\n{\bf  Some  abbreviations:}  For  short  we  write  f.d.  for  finite  dimensional,
s.a.  for  self-adjoint,  c.p.  for  completely  positive,  c.b.  for  completely  bounded. We denote by $CB(E,F)$ the space of c.b. maps
between two o.s. $E,F$ equipped with the c.b. norm.
  We  reserve  the  notation  $E\otimes  F$  for  the  algebraic  tensor  product
  of  two  linear  spaces. We denote by $Id_E$ the identity 
map on $E$.\\
We refer to   \cite{ER,P4} for background on operator spaces.
  \section{Nuclear  pairs}\label{nuc}

We  start  by  a  few  general  remarks  around  nuclearity  for  pairs.
    \begin{dfn}  A  pair  of  $C^*$  algebras  $(A,B)$  will  be  called  a  nuclear  pair
    \index{nuclear  pair}  if
    $$A  \otimes_{\min}  B=A  \otimes_{\max}  B,$$
    or  equivalently  if  the  min-norm  is equal to  the max-norm     
  on  the  algebraic  tensor  product  $A\otimes  B$.
  \end{dfn}
    \begin{rem}\label{da1}  If  the  min-norm  and  the max-norm  are  equivalent  on  $A\otimes  B$,
  then  they  automatically  are  equal.  
    \end{rem}
        \begin{rem}\label{da20}  
        Let  $A_1\subset  A$  and  $B_1\subset  B$  be  $C^*$-subalgebras.
        In  general,  the  nuclearity  of  the  pair  $(A,B)$  does  \emph{not}  imply  that  of  
        $(A_1,B_1)$.  
        This  ``defect"  is  a  major  feature  of  the  notion  of  nuclearity.
        However,  if    $(A_1,B_1)$  admit  contractive  c.p.  projections  (conditional  expectations)
        $P:  A  \to  A_1$  and  $Q:  B  \to  B_1$  
        then    $(A_1,B_1)$  inherits  the  nuclearity  of    $(A,B)$.  
        More  generally,  this  holds  if  we  only  have  approximate  versions  of  $P$  and  $Q$.
        For  instance,  if  $A_1$  and  $B_1$  are  (closed  s.a.)  ideals  in  $A$  and  $B$
        then  the  nuclearity  of  the  pair  $(A,B)$  does    imply  that  of  
        $(A_1,B_1)$.
            \end{rem}
Recall  that  $A$  is  called  nuclear
\index{nuclear  $C^*$-algebra}  if  $(A,B)$  is  nuclear  for  all  $B$.\\
The  basic  examples    of  nuclear  $C^*$-algebras  
include  all  commutative  ones,
the  algebra  $K(H)$  of  all  compact  operators  on  an  arbitrary  Hilbert  space  $H$,  $C^*(G)$
for      all  amenable  discrete  groups  $G$  and  the  Cuntz  algebras.

We  wish  to  single  out        two  fundamental  examples
$$\mathscr{B}=B(\ell_2)  \quad  {\rm  and  }  \quad\mathscr{C}=C^*(\F_\infty).$$
Recall  that  every  separable  unital  $C^*$-algebra  embeds  in  $\mathscr{B}$  and  is  a  quotient  of  $\mathscr{C}$. 
Moreover $\mathscr{B}$ is injective while $\mathscr{C}$ is in some sense projective (see Remark \ref{tr1}). Neither  $\mathscr{B}$  nor  $\mathscr{C}$  is  nuclear,  nevertheless  :

\begin{thm}[Kirchberg  \cite{Kiuhf}]\label{kir}  The  pair  $(\mathscr{B}  ,\mathscr{C})$  is  nuclear.\\
More  generally,  for  any  free  group  $\F$  and  any
$t\in  B(H)  \otimes  C^*(\F)$  or  $t\in  C^*(\F)  \otimes  B(H)$  we  have  $\|t\|_{\min}=\|t\|_{\max}$.
\end{thm}

A  simpler  proof  appears  in  \cite{P}  (or  in  \cite{P4},  or  now  in  \cite{P6}).    

Since  Kirchberg  \cite{Kir}  showed  that  a  $C^*$-algebra  $A$
has  Lance's  WEP if and only if  the  pair  $(A,\mathscr{C})$  is  nuclear,
we  took  the  latter  as  our  definition  of  the  WEP.  Kirchberg  \cite{Kir}  also  showed
that  $A$
has  a  certain  local  lifting  property  (LLP) if and only if  the  pair  $(A,\mathscr{B})$  is  nuclear.  We  again  take  the  latter  as  the  definition  of  the  LLP.
With  this  terminology,  Theorem  \ref{kir}  admits  the  following  generalization:
\begin{cor}  Let  $B,C$  be  $C^*$-algebras.  If  $B$  has      the  WEP  and  $C$  the  LLP  then  the  pair
$(B,C)$  is  nuclear.
\end{cor}

  In  \cite{JP}  it  was  shown  that  $\mathscr{B}$  failed  the  LLP,
or  equivalently  that  the  pair  $(\mathscr{B},\mathscr{B}  )$  was  not  nuclear,
which  gave  a  negative  answer  to  one  of  Kirchberg's  questions  in  \cite{Kir},  namely  whether
the  WEP  implies  the  LLP.
However,  the  following  major  conjecture  (equivalent  to  the  converse  implication)  remains  open:\\
{\bf  Kirchberg's  conjecture  :}  The  pair  $(\mathscr{C},\mathscr{C}  )$  is  nuclear,
or  equivalently  $\mathscr{C}$  has  the  WEP.

Kirchberg  showed  at  the  end  of  \cite{Kir}  that  this  conjecture  is  equivalent
to  the  Connes  embedding  problem  whether  any  finite  von  Neumann  algebra
embeds  in  an  ultraproduct  of  matrix  algebras.

The  Kirchberg  conjecture  asserts  that  the  min  and  max  norms
coincide  on  $    \mathscr{C}\otimes  \mathscr{C}      $.
More  recently  in  \cite[Th.  29]{Oz2},  Ozawa      
      proved  that  to  confirm  the  Kirchberg  conjecture  it  suffices  to  show
      that  they  coincide  on  $E^1_n  \otimes  E^1_n$  \emph{for  all}  $n\ge  1$,
      where  $E^1_n$  is  the  span  of  the unit and  $n-1$  first  free  unitary  generators
      of  $\mathscr{C}  $.

  \begin{rem}\label{rus1}  Recall  $A^{op}\simeq  \ovl  A$,  where  $A$  is  the  complex  conjugate  of  $A$.
  By  an  unpublished  result  of  Haagerup  (see  \cite{P6}  for    complete  details)
  a  $C^*$-algebra  $A$  has  the  WEP
 if and only if  the  min  and  max  norms  coincide  on  the  set  of  tensors  in  $\ovl  A  \otimes      A$
  of  the  form  $\sum  \ovl{a_j}  \otimes  a_j$  (we  call  these  positive  definite  tensors  in  \cite{P6}).
    \end{rem}
    
\begin{rem}\label{ap1}  Kirchberg's  construction  in \cite[Th. 1.2]{Kir} of  a  $C^*$-algebra
  $A$ satisfying  \eqref{e13}  is   quite  difficult  to  follow.  
  He   uses to start with a $C^*$-algebra $Q$ that is a quotient of a WEP $C^*$-algebra (QWEP in short).
Then his    $A$  is a (quasidiagonal) extension by the compact operators $\cl K$ of the cone algebra $C(Q)$ of $Q$. The non-nuclearity (actually non-exactness)
of $A$ follows from the fact that for a well chosen $Q$ the short exact sequence 
$\cl K \to A \to C(Q) $ does \emph{not} locally split in his sense.
His $A$ must have the WEP (see Remark \ref{rus1}),
and hence $C(Q)$ and $Q$ are necessarily QWEP  but it seems
unclear whether a suitable choice of $Q$ might lead to
one  for which $A$ has LLP.  Note however 
that $Q$ must be chosen to fail the LLP, otherwise (see \cite[p. 454]{Kir})
  the exact sequence will locally split.
  A  much  clearer  presentation  (unfortunately
without  the  full  details) of his arguments  is  sketched    in  Remark  13.4.6  of  Brown  and  Ozawa's  remarkable  book  \cite{BO}.
    \end{rem}

In  the  next  two  sections  we  gather  some  known  facts  on  the  WEP  and  the  LLP,
that  were  probably  all  known  in  some  form  to  Kirchberg  at  the  time  of  \cite{Kir}.
Since  we  use
    reformulations  best  suited  for  our      construction,  we  include  proofs.
  We  refer  the  reader  to  Ozawa's  concise  survey  \cite{Oz1}
  or  to  our  much  longer  exposition  in  \cite{P6}    for  more  information.

\section{The WEP}\label{wep}

      \n    We  define  the  WEP  for  a  $C^*$-algebra  $A$
  by  the  equality        $A\otimes_{\min}  {\mathscr{C}}=A\otimes_{\max}  {\mathscr{C}}$,  where
  ${\mathscr{C}}$  is  the  full  (or  maximal)  $C^*$-algebra  of  the  free  group
  $\F_\infty$.  
  Kirchberg    showed  that  this  property  is  equivalent  to  a  weak  form  of  extension  property
  (analogous  to  that  of  $L_\infty$  in  Banach  space  theory),  a variant of injectivity that  had  been  considered    by    
      Lance  \cite{Lan}.
      
      Assume  $A\subset  B(H)$  as  a  $C^*$-subalgebra.
      Then  $A$  has  the  WEP
     if and only if  there  is  a    contractive  projection  $P:  B(H)^{**}  \to  A^{**}$.
      
        Equivalently,  this  holds if and only if  there  is  a      contractive
        linear  map  $T:  B(H)  \to  A^{**}$  such  that  $T(a)=a$  for  any  $a\in  A$,
        or  in  other  words  such  that  $T_{|A}$  coincides  with  the  canonical  inclusion
        $i_A:  A  \to  A^{**}$.  Note  that  when  it  exists
        the  contractive  projection  $P$  is  automatically
        completely  contractive  and  completely  positive  by  Tomiyama's  well  known  theorem.  
      This  leads  to  the  following  simple  (known)  criterion
      which,  being  almost  purely  Banach  space  theoretical,  
      will  be  particularly  well  adapted  to  our  needs.
      
      We  denote  here  by  $\ell_1^n  $  the  operator  space  dual  of  $\ell_\infty^n$.
      One  nice  realization  of  $\ell_1^n  $  
      can  be  given  inside  $\C$  :  just  let
      $E_1^n  ={\rm  span}[1,U_1,\cdots,U_{n-1}]  \subset  \C$  where  $(U_j)$  are  the  free  unitary  generators
      of  $\C$,  then  $\ell_1^n  \simeq    E_1^n  $  completely  isometrically.
      Equivalently we could take instead the span of $\{U_1,\cdots,U_n\}$.
      Note   
that  $\|v\|_{cb}=\|v\|$  for  any  $v$  defined  on  $\ell_1^n$.  This  is  known  as  the  maximal  operator  space  structure  of  $\ell_1^n$ (see e.g. \cite[p. 183]{P4}).

      \begin{pro}\label{p1}  A  $C^*$-algebra        $A\subset  B(H)$  has  the  WEP
      if  (and  only  if)  
      for  any  $n\ge  1$  and  any  subspace  $S\subset  \ell_1^n$
          any  linear  map  $u:S  \to  A$  admits
          for  each  $\vp>0$    an  extension
          $\tilde  u:  \ell_1^n\to  A$  with
          $$\|      \tilde  u\|_{cb}  \le  (1+\vp)\|          u\|_{cb}.$$
              \end{pro}
    \begin{proof}  This  is  a  well  known  application  of      Hahn-Banach.
    Let  $B$  be  another  $C^*$-algebra.
    Note  the  isometric  identity
    \begin{equation}  \label{e1}
    B(B  ,  A^{**})=  (B    \otimes_{\wedge}  A^*)^*,\end{equation}
    where  $B  \otimes_{\wedge}  A^*$  denotes  the  normed  space  $B\otimes  A^*$  (algebraic)
    equipped  with  the
    projective  norm,      denoted  by  $\|\  \|_{B    \otimes_{\wedge}  A^*}  $.  Let  $X=  B(H)  \otimes    A^*$  and  $Y=A  \otimes    A^*$
    so  that  $Y\subset  X$.
Consider  the  assertion    that  the  inclusion    \begin{equation}  \label{e2}(Y,\|\  \|_{A  \otimes_{\wedge}  A^*}  )  \subset  (X,\|\  \|_{B(H)  \otimes_{\wedge}  A^*}  )\end{equation}
    is  isometric.  We  claim  that  if  this  holds  then    $A$  has  the  WEP.
    Indeed,  assume  that  \eqref{e2}  is  isometric.
      Consider  the  linear  form
    $f:  Y  \to  \CC$  defined  by  $f(a\otimes  \xi)=\xi(a)$  (which  corresponds  through  \eqref{e1}  with  $B=A$
    to  $i_A:A\to  A^{**}$).
    By  Hahn-Banach,  $f$  extends  to  a  linear  form  $g:  (X,\|\  \|_{B(H)  \otimes_{\wedge}  A^*}  )  \to  \C$  with  $\|g\|\le  1$.
    Let  $T:  B(H)  \to  A^{**}$  be  the  map  associated  to  $g$  via  \eqref{e1}.
    Then  $\|T\|  \le  1$  and  
    the  fact  that  $g$  extends  $f$  is  equivalent  to
    $T_{|A}=i_A$,  so  that  $A$  has  the  WEP.
To  complete  the  proof  of  the  if  part
it  suffices  to  show  that  the  extension  property
in  Proposition  \ref{p1}  implies
that  \eqref{e2}  is  isometric.
This  is  easy  to  show  using  the  factorization  of
the      mappings  $v:  A\to  B(H)$  corresponding  to  an  element  $x$  in  the  open  unit  ball
of  $(X,\|\  \|_{B(H)  \otimes_{\wedge}  A^*}  )$.
Such  a  $v$  can  be  written  as
$v=  U  V$,  where    $V:  A  \to  \ell_1^n$  has      nuclear  norm  $1$
and  $\|U:  \ell_1^n  \to  B(H)\|    <1$.      
If  it  so  happens  that    $x\in  A\otimes    A^*$  then  $v(A)=UV(A)\subset  A$.
Let  $S=V(A)\subset  \ell_1^n$  and    $u=U_{|S}:  S  \to  A$.
Note  $\|U  \|_{cb}=\|U  \|$,  and  hence
$\|u\|_{cb}  \le  \|U\|_{cb}  <1$.
Let  $\tilde  u$  be  as  in  the        extension  property
in  Proposition  \eqref{p1},  then  the  factorization
$v=  \tilde  u  V$  now  shows  that  $v$  corresponds  to  an  element  $x$  in  the  open  unit  ball
of  $(Y,\|\  \|_{A  \otimes_{\wedge}  A^*}  )$.  Thus  \eqref{e2}  is  isometric.
This  proves  the  if  part.\\
Conversely,  if  $A$  has  the  WEP,  by  the  injectivity  of  $B(H)$
any  $u:  S  \to  A$  extends  to  a  map      $u^1:  \ell_1^n  \to  A^{**}$  with  $\|u^1\|_{cb}=
\|u\|_{cb}$  and  ${u^1}_{|S}=i_A  u$.  Since  $\ell_\infty^n(A^{**})=\ell_\infty^n(A)^{**}$  isometrically,
there  is  a  net  of  maps  $u_i:  \ell_1^n  \to  A$  with  $\|u_i\|_{cb}\le  1$
tending  pointwise  $\sigma(A^{**},  A^*)$  to  $u^1$.
Then  $(u_i-u)_{|S}  $  tends  pointwise  $\sigma(A  ,  A^*)$  to  $0$,
and  by  Mazur's  theorem  passing  to  convex  combinations  we  obtain  a  net  
$u_i':  \ell_1^n  \to  A$  with  $\|u'_i\|_{cb}\le  1$  such    that  $(u'_i-u)_{|S}  $  tends
pointwise  to  to  $0$  in  norm.  
Thus  for  $i$  large  enough  $u_i'$  is  ``almost"  
the  desired  extension  of  $u$.  Then  for  each  $\vp>0$    a  simple  perturbation  argument
(see  e.g.  \cite[p.  69]{P4})
gives  us  a  true  extension  $\tilde  u$  as  in  Proposition  \ref{p1}.
      \end{proof}
            \begin{rem}\label{rr5}  In  the  preceding  situation,  
            assume  $A\subset  B(H)$.  A  linear  map  $u:  S  \to  A$
            satisfies  $\|u\|_{cb}\le  1$ if and only if  it  admits  an  extension
            $\tilde  u:  \ell_1^n  \to  B(H)$  with  $\|\tilde  u\|\le  1$.
            This  is  immediate  by  the  injectivity  of  $B(H)$
            and  the  equality  $\|\tilde  u\|=\|\tilde  u\|_{cb}    $.\\
            This  allows    us  to  view  the  extension  property  in  Proposition  \ref{p1}  as  a  Banach  space  theoretic  property  of  the  inclusion  $A\subset  B(H)$  like  this:
            any  $u$  that  extends  to  a  contraction  into  $B(H)$
            extends  to  a  map  of  norm  $\le  1+\vp$  into  $A$.
                  \end{rem}  
      \begin{rem}\label{R1}  The  interest  of  Proposition  \ref{p1}
      is  that  the  apparently  weak  form  of  extension  property  considered  there  suffices  to  imply  the  WEP.
      But  actually,  any  $A$  with  WEP  satisfies  a  stronger  extension  property,  as  follows:\\
    {\it  Let  $C$  be  a  separable  $C^*$-algebra  with  the  LLP  and  let  $A$  be  another  one  with  the  WE{P}.    Then  for  any  finite  dimensional  subspace  $E\subset  C$  and  any  $\vp>0$,  any  $u\in  CB(E,A)$  admits  an  extension  
  $\tilde{u}\in  CB(C,A)$  such  that  $\|\tilde{u}\|_{cb}\leq(1+\vp)\|  u\|_{cb}$.}
  See  \cite{Kir}  or  \cite[Th.  20.27]{P}  for  full  details.
$$\xymatrix{&C\ar@{-->}[dr]^{  \tilde  u  }  \\
&  E\ar@{^{(}->}[u]  
  \ar[r]^{  u  \quad  }  &  A}$$
      \end{rem}

  \section{The  LLP}\label{llp}

      \n    We  define  the  LLP  for  a  $C^*$-algebra  $A$
  by  the  equality        $A\otimes_{\min}  {\mathscr{B}}=A\otimes_{\max}  {\mathscr{B}}$,  where
  ${\mathscr{B}}=B(\ell_2)$.  
  Kirchberg    showed  that  this  property  is  equivalent  to  a  
  certain  local  lifting  property  (analogous  to  that  of  $L_1$  in  Banach  space  theory),
  which  has  several  equivalent  forms,  one  of  which  as  follows:
  
  \begin{pro}\label{p3}  A  $C^*$-algebra  $A$  satisfies  $A\otimes_{\min}  {\mathscr{B}}=A\otimes_{\max}  {\mathscr{B}}$
  if and only if   for  any  $*$-homomorphism  $\pi:  A  \to  C/\I$  into  a  quotient  $C^*$-algebra,
  for  any  f.d.  subspace  $E  \subset  A$  and  any  $\vp>0$  the  restriction  $\pi_{|E}$  admits  a  lifting
  $v:  E  \to  C$  with  $\|v\|_{cb}\le  (1+\vp)$.
  \end{pro}
  \n Note:  Actually   when  $A$  has  the  LLP  the  preceding  local  lifting  even  holds  with  $\vp=0$ (see the proof of Proposition \ref{p3}).
  \begin{rem}\label{tr1}  According  to  Kirchberg  \cite{Kiuhf}  a unital   $C^*$-algebra  $A$  has  the  lifting  property  (LP)  if  any  unital c.p.   $\pi:  A  \to  C/\I$  as  above  admits  a  (global)  unital  c.p.  lifting, and if $A$ is not unital, it is said to have the LP if its unitization does.
   He  proved  that  $\C$  has  the  LP (see \cite{Oz1} or \cite[p. 376]{BO} for a proof).  By  known  results,  it  follows  that  $C\otimes_{\min}  \C$
  has  the  LP  whenever  $C$  is  nuclear  and  separable.  Roughly  this  can  be  checked  using  the  CPAP
  of  $C$, Kirchberg's theorem that $M_n(\C)$ has the LP for any $n$ (see \cite[p. 59]{Kiuhf})  and  the  fact  (due  to  Arveson \cite{[Ar4]})  that  the  set  of  liftable  maps  
  on  separable  $C^*$-algebras  is  pointwise  closed. 
  Actually, Kirchberg observed in \cite{Kir} that if his conjecture that 
  $\C$ has the WEP is correct, then the LLP implies the LP in the separable case.
  Kirchberg's LP (as we just defined it) implies
  that any completely contractive c.p. map $u: A \to C/\I$ admits a completely contractive lifting
  from $A$ to $C$, but the converse does not seem clear, although the analogous converse does hold  for the LLP.
    \end{rem}

\begin{lem}\label{l4}  Assume  that  an  isomorphism  $\pi:  A  \to  C/\I$  is  locally  liftable,  meaning  that  it
  has  the      property  considered  in    Proposition  \ref{p3}.  Let  $B$  be  any  $C^*$-algebra.\\
Then      $(C,B)$    nuclear  $\Rightarrow$  $(A,B)$    nuclear.\\
In  particular  if $C$  has the LLP  then   $A$  has the LLP.  \end{lem}
\begin{proof}
Indeed,  let  $t\in  E\otimes  B$  with  $E\subset  A$  f.d.  with  $\|t\|_{\min}=1$.
Let  $q:  C  \to  C/I  $  be  the  quotient  map.
Then    if  $(C,B)$    is  nuclear  
$$\|  (v\otimes  Id)(t)\|_{C\otimes_{\max}  B}=\|  (v\otimes  Id)(t)\|_{C\otimes_{\min}  B}\le  1+\vp$$
and  hence
$\|  (\pi\otimes  Id)(t)\|_{C/I\otimes_{\max}  B}=\|  (qv\otimes  Id)(t)\|_{C/I\otimes_{\max}  B}    \le  1+\vp$.
Since  $\pi$  is  an  isomorphism  we  obtain  $\|t\|_{A\otimes_{\max}  B}=1$,  and  hence
$(A,B)$  is  nuclear.  \end{proof}

\begin{proof}[Proof  of  Proposition  \ref{p3}]  
Let  $A$  be  unital.  Then  $A=C^*(\F)/I$  for  some  free  group  $\F$,
and  $C^*(\F)$  has  the  LLP.
In  general  let  $\tilde  A$  be  the  unitization,  let  $q:  C^*(\F)  \to  \tilde  A$
be  a  surjective  $*$-homomorphism,  and  let  $C=q^{-1}(A)\subset  C^*(\F)$.
Since  $C$  is  an  ideal  in  $C^*(\F)$,  it  still  has  the  LLP.
By  Lemma  \ref{l4}  with  $\pi:  A  \to  C/\ker(q)$,  the  if  part  follows.\\
Conversely,  if  $A$  has  the  LLP,  so  does  its  unitization  $\tilde  A$  (see  Remark  \ref{da20}).
If  $\pi:  A  \to  C/I$  is  as  in  Proposition  \ref{p3},  then  $\pi$  extends  to  a  unital
$*$-homomorphism    $\tilde  \pi:  \tilde  A  \to  \tilde  C/I$  
and  it  is  easy  to  check  that  if  $\tilde  \pi$  is  locally  liftable  then
$\pi$  also  is.  Thus  to  prove  the  only  if  part  we  may  assume  that  $A,C$  and  $\pi$  are  all  unital.
Again  we  write  $A=C/I$  with  $C=C^*(\F)$.  Let  $E\subset  A  $  be  a  f.d.  subspace.  
The  inclusion  map  $E\to  A$  corresponds  to  a  tensor  $t\in  E^*  \otimes  A$
with  $\|t\|_{\min}=1$  when  $E^*$  is  equipped  with  its  dual  operator  space  structure  (see  \cite{ER,P4}).
We  may  assume  $E^*\subset  \B$  completely  isometrically.
Viewing  $t\in  \B  \otimes  A$  the  LLP  assumption  implies  $\|t\|_{  \B  \otimes_{\max}  A}=1$.
Since  ${  \B  \otimes_{\max}  A}={  \B  \otimes_{\max}  C}/{  \B  \otimes_{\max}  I}$.  
The  tensor  $t$  admits  a  lifting  $t\in  E^*  \otimes  C$  with  $\|t\|_{  \B  \otimes_{\max}  C}\le 1+\vp$.
A  fortiori,  we  have  $\|t\|_{  E^*\otimes_{\min}  C}<1+\vp$,  and  the
linear  map  $\tilde  u:  E  \to  C$  associated  to  $t$  gives  us  the  desired  local  lifting.
With more effort (see \cite[p. 45]{P4}) one can show the same for $\vp=0$. 
\end{proof}

Let    $E,F$  be  operator  spaces.
Recall
\begin{equation}\label{ee11} 
d_{cb}(E,F)=\inf  \{\|u\|_{cb}\|u^{-1}\|_{cb}\}\end{equation}
where  the  infimum  runs  over  all  complete  isomorphisms  $u:E  \to  F$  whenever
$E,F$  are  completely    isomorphic  
(e.g.  if  $E,F$  are  of  the  same  finite  dimension),
and  $d_{cb}(E,F)=\infty$  if  they  are  not  completely  isomorphic.
Definition \ref{d17}  below is just the obvious analogue
   of the notion of a Banach space finitely representable (in short f.r.) in another one.
\begin{dfn}\label{d17}
Let  $A,C$  be     $C^*$-algebras (or operator spaces).
We will say that $A$ ``locally embeds" in $C$
if
for any $\vp>0$ the following property holds: for any f.d. $Z \subset A$ there is
a f.d. $Z' \subset C$ such that $d_{cb}(Z,Z')\le 1+\vp$.
\\
We say that $A$ and $C$ are ``locally equivalent" if
each one locally embeds in the other.
\end{dfn}
 \begin{rem} Concentrating on the case $C=\C$,
 let us denote (as  in  \cite[p.  343]{P4})    for any o.s. $A$
$$d_{S\C}(A)=\sup  \{\inf  \{d_{cb}(Z,Z')\mid  Z'\subset  \C\} \mid  Z\subset  A, \dim(Z)<\infty\}.$$
One defines similarly $d_{S\cl C}(A)$ for any $C^*$-algebra $\cl C$.\\
 Then $A$ locally embeds in $\C$ means that $d_{S\C}(A)=1$.
The following  formula
 was proved in \cite[Th. 4.5]{JP} (see also \cite[p. 349]{P4}
 or \cite[Cor. 20.6]{P6}): for any o.s. $A\subset B(H)$
 we have
  \begin{equation}  \label{us10}
  d_{S\C}(A)
 =
 \sup\{\|t\|_{B(H)\otimes_{\max} \B }/ \|t\|_{B(H)\otimes_{\min} \B } \mid t\in A \otimes \B \}.\end{equation} 
 When $A$ is a $C^*$-algebra, the basic properties
 of $C^*$-norms imply that 
if 
$d_{S\C}(A) \le 1+\vp$ for some $\vp>0$
then $d_{S\C}(A)=1$. 
\end{rem}
  \begin{rem}\label{rus3}
  Note that when $A\subset B(H)$, we have obviously   $\|t\|_{B(H)\otimes_{\max} \B }
  \le \|t\|_{A\otimes_{\max} \B }$
  for any $t\in A \otimes \B$. Therefore,
     \eqref{us10} implies that $ d_{S\C}(A)
 =1$ for any $A$ with the LLP.
 In other words, any $A$ with the LLP locally embeds in $\C$.
 \end{rem}
The  next  result  is  related  to  \cite{JP}.  In  the  latter,
it  was  shown  that  WEP  $\not\Rightarrow$  LLP,  and  at  the  same  time  that  there  are  
$A$'s     that  do not locally embed in $\C$.    \begin{pro}\label{p2}  Let  $A$  be  a  $C^*$-algebra    that  locally embeds in $\C$.
    If        $A$  has  the  WEP
    then  it  has  the  LLP.
  \end{pro}
  \begin{proof}
  We  will  use  several  results  from  \cite{Ha}  (see  also    \cite[chap.  6]{P6}  for  a  detailed  presentation).
  Assume    that $A$  locally embeds in $\C$ and  has  the  WEP.  Let  $t\in  A  \otimes  \B$.
  We  will  show  that  $\|t\|_{\max}=\|t\|_{\min}$.
  Let  $Z\subset  A$  be  a  f.d.  subspace  such  that  $t\in  Z  \otimes  \B$.
  For  any  $\vp>0$  there  is  $Z'\subset  \C$    and  an  isomorphism  $v:Z'  \to  Z  \subset  A$  with
  $\|v\|_{cb}\|v^{-1}\|_{cb}<1+\vp$.  Using  the  factorization  of  the  canonical  inclusion  $i_A:  A  \to  A^{**}$
    through  some  $B(H)$  (which  is  one  form  of  the  WEP)
    we  find  an  extension  $\tilde  v:  \C  \to  A^{**}$  of  $v$
        such  that
    $\|\tilde  v\|_{dec}  =\|v\|_{cb}$  (here  we  use  the  dec-norm  of  \cite{Ha}  and  the  fact  due  to  Haagerup  that  
    the  dec-norm  and  the  cb-norm  coincide  for  maps  with  range  $B(H)$).
  Now  we  have  $  (i_A  \otimes  Id)t=(\tilde  v  v^{-1}\otimes  Id)  t$  and  hence  (again  using  \cite{Ha})
  $$\|(i_A  \otimes  Id)t\|_{A^{**}  \otimes_{\max}  \B  }
  \le  \|\tilde  v\|_{dec}    \|(v^{-1}\otimes  Id)  t\|_{\C  \otimes_{\max}  \B  }  \le  \|    v\|_{cb}    \|(v^{-1}\otimes  Id)  t\|_{\C  \otimes_{\max}  \B  }.$$
  By  Theorem  \ref{kir},  
  $\|(v^{-1}\otimes  Id)  t\|_{\C  \otimes_{\max}  \B  }=\|(v^{-1}\otimes  Id)  t\|_{\C  \otimes_{\min}  \B  }$
and  hence  we  find
$$\|(i_A  \otimes  Id)t\|_{A^{**}  \otimes_{\max}  \B  }
  \le  \|    v\|_{cb}    \|(v^{-1}\otimes  Id)  t\|_{\C  \otimes_{\min}  \B  }  \le  \|    v\|_{cb}  \|    v^{-1}\|_{cb}  \|t\|_{A    \otimes_{\min}  \B  }\le  (1+\vp)    \|t\|_{A    \otimes_{\min}  \B  }.$$
  But  by  an  easy  and  well  known  argument  we  have
  $\|(i_A  \otimes  Id)t\|_{A^{**}  \otimes_{\max}  \B  }=\|  t\|_{A    \otimes_{\max}  \B  }$  for  any  $t\in  A\otimes  \B$,
  thus  since  $\vp>0$  is  arbitrary  we  conclude  $\|t\|_{\max}\le  \|t\|_{\min}$.
        \end{proof}  
  
\section{Outline}
\label{out}
  We  follow  the  general  strategy  in  \cite{P0} (see also \cite{BoP}).
We  construct  a  sequence  of  operator  spaces  $E_n  \subset  E_{n+1}$
such  that  
for  any  $S  \subset  \ell_1^n$,  any  $u:  S  \to  E_n$  admits  an  extension  (or  an  approximate  extension)
$\tilde  u  \to  E_{n+1}$  
but  into  the  larger  space  $E_{n+1}$,  with  $\|\tilde  u\|_{cb}  \approx      \|u\|_{cb}$,  as  in  the  following  diagram.

$$\xymatrix{&\ell_1^n\ar[r]^{  \tilde  u  \quad  }  &  E_{n+1}    \\
&  S\ar@{^{(}->}[u]  
  \ar[r]^{  u  \quad  }  &  E_n\ar@{^{(}->}[u]}$$
  
  The  rough  plan  is  then  to  start  from  a  space  $E_1$  such  that  $d_{S\C}(E_1)=1$
  and  find  successive  spaces  $E_n$  while  maintaining  the  condition  $d_{S\C}(E_n)=1$  for  all  $n$.
  Then  the  idea  is      that  the  union  $X=\ovl{\cup  E_n}$  will  satisfy  the  extension  property
  in  Proposition  \ref{p1}  (that  is  equivalent  to  the  WEP  for  a  $C^*$-algebra),
while      the  condition  $d_{S\C}(E_n)=1$  for  all  $n$  will  imply
$d_{S\C}(X)=1$  and  hence  by  Proposition  \ref{p2}  the  purported  WEP  of    $X$  
  would  imply  its  LLP.  Just  starting  from  a  f.d.  operator  space  $E_1$
  with  exactness  constant  $>1$  will  ensure  that  $X$
  if  it  is  a  $C^*$-algebra  is  not  exact,  and  hence        not  nuclear.
  
Since  we  can  embed  $E_n$  into  an  injective  object,  namely  $B(H)$
(and  $\ell_\infty$  in  the  Banach  space  case)  it  is  easy  to  do  the  first  step  $n=1$  
and  to  find  some  ``big"  space  $E_2$  satisfying  the  extension  but
the  difficulty  is  to  find  $E_2$  still  such  that  $d_{S\C}(E_2)=1$.  The  latter  expresses  that
$E_2$  remains  relatively  ``small".
In  the  Banach  space  analogue  of  \cite{P0}  this  is  the  main  problem:  there  bounded  cotype
2  constants  are  the  key  tool  that  replaces  $d_{S\C}(E_n)=1$;  note  however
  that  the  construction  there  is  isomorphic  (with  uniformly  bounded  constants),  as  opposed  to  the  present  ``almost"  isometric  one  (i.e.  with  constants  asymptotically  tending  to  1).
  However,  it  is  actually  possible  to  essentially  proceed  and  
  maintain  the  condition  $d_{S\C}(E_n)=1$  for  all  $n$,  using  the  operator  space  analogue
  of  the  construction  in  \cite{P0},  but  in  the  completely isometric  setting,  with  all  relevant  constants
  equal  to  $1$.  This  led  to      operator  space  versions  of  our  main  result,
  that  we  obtained  already  a  few  years  ago  (we  gave  a  talk  on  this  at  MSRI  in  the  Fall  2016).  We plan to write the details in a separate paper.
  
  But this seemed like a dead end because it gave no clue how  to  arrange  so  that  the  union  $\ovl{\cup  E_n}$
  be  not  only  an  operator  space  but  a  genuine  $C^*$-algebra.  
  For  this  we  need  to  produce  embeddings
  $T_n  :  E_n  \to  E_{n+1}$  that  are    multiplicative,  or  at  least      
  close  to  multiplicative,  in  such  a  way  that  in  the  limit  we  obtain  
  an  algebra  and  multiplicative  maps. This  is now overcome  by    
  a  quite  different  construction  of  $E_{n+1}$  given  $E_n$  based  on  Lemma
  \ref{l2}  below. 
  We construct our example as an inductive limit of f.d. operator spaces,
  with linking maps that are close to multiplicative in some suitable sense.
  Similar asymptotic morphisms 
  already appear in the $E$-theory of Connes and Higson \cite{CoH} but
  our use of them seems unrelated.
  There are many examples  of inductive limits of 
            $C^*$-algebras  in the literature, see for example \cite{BK97}. 
  See also \cite[p. 465]{Bla}, \cite[ch. 6]{LauR} or \cite{Lor}
  for general background on inductive limits.

\section{Approximately multiplicative maps}

We will need to use liftings that are approximately multiplicative.
This will be provided by the use of the cone algebra $C_0(\C)$, but it
seems worthwhile to first
  collect some
            basic general facts about almost multiplicative mappings.
            The latter facts are known (the ideas go back at least to \cite{CoH})
            but do not seem to have been recorded
            in the literature.
                       \begin{dfn}\label{rd2} Let $B,B_0$ be $C^*$-algebras.
            Let $ \cl E_0\subset B$ be a self-adjoint subspace   and let $\vp_0>0$. A linear map ${\psi}: \cl E_0 \to B_0$ 
            (we will restrict in (iii) to the s.a. case for simplicity)
            will be called an $\vp_0$-morphism
            if  \begin{itemize}
             \item[{\rm (i)}] $\|{\psi}\|\le 1+\vp_0$,
             \item[{\rm (ii)}]
        for any $x,y\in \cl E_0$ with $xy\in \cl E_0$ we have
          $\|{\psi}(xy) -{\psi}(x){\psi}(y)\|\le \vp_0 \|x\|\|y\|,$
          \item[{\rm (iii)}]
          ${\psi}$ is self-adjoint i.e. for any $x\in \cl E_0$    we have        ${\psi}(x^*) ={\psi}(x)^*$.
           \end{itemize}
           \end{dfn}
          \begin{rem} 
          Let ${\psi}_0: \cl E_0 \to B_0$ be a linear map.
          Let $\cl E_1 \subset B_0$ be another self-adjoint subspace such that 
          ${\psi}_0(\cl E_0) \subset \cl E_1$ and
          $ {\psi}_0(\cl E_0){\psi}_0(\cl E_0) \subset \cl E_1 $.
        Let
           $B_1$ be another
        $C^*$-algebra.
          If ${\psi}_0: \cl E_0 \to B_0$ is an $\vp_0$-morphism
          and ${\psi}_1: \cl E_1\to B_1$ an $\vp_1$-morphism ($\vp_1>0$), then
          ${\psi}_1{\psi}_0: \cl E_0 \to B_1$ is an $\vp_0\|{\psi}_1\|+\vp_1\|{\psi}_0\|^2$-morphism,
          and hence a
          $\d_1$-morphism with
         \begin{equation}  \label{ee3'}\d_1\le \vp_0(1+\vp_1) +\vp_1 (1+\vp_0)^2
         .\end{equation}          This is immediate
          since
          $$ {\psi}_1{\psi}_0(xy)-{\psi}_1{\psi}_0(x){\psi}_1{\psi}_0(y)=
          {\psi}_1({\psi}_0(xy)-{\psi}_0(x){\psi}_0(y)) +{\psi}_1 ({\psi}_0(x){\psi}_0(y)) -{\psi}_1{\psi}_0(x){\psi}_1{\psi}_0(y).$$
          More generally if
          we have a similar composition ${\psi}_n\cdots {\psi}_1{\psi}_0$
          with ${\psi}_j$ being an $\vp_j$-morphism, 
          then ${\psi}_n\cdots {\psi}_1{\psi}_0: \cl E_0 \to B_n$ has norm 
          $\le (1+\vp_n)\cdots(1+\vp_1)(1+\vp_0)$.
          Let  $p_n= (1+\vp_n)\cdots (1+\vp_0)$.
          Using  \eqref{ee3'}, a simple induction shows that  ${\psi}_n\cdots {\psi}_1{\psi}_0$ is a $\d_n$-morphism
          with $\d_n$ satisfying :
          $$\d_n\le \d_{n-1} (1+\vp_n) + \vp_{n}  p_{n-1}^2.$$
          Thus the number $\d_n'=\d_n p_n^{-1}$
          satisfies
          $$\d'_n\le \d'_{n-1}+   \vp_{n}  p_{n-1}^2/p_n\le  \d'_{n-1}+   \vp_{n}  p_{n-1} .$$
          Thus if we assume that   $\sum \vp_n<\infty$ so that $\sup_n p_n= c<\infty$ we find
          $$\d_n'\le \d_0' +c\sum\nl_1^n \vp_{k} $$
          and hence for any $n\ge 1$
           \begin{equation}  \label{ee3}\d_n\le p_n (\d_0' +c\sum\nl_1^n \vp_{k} )\le c \vp_0+c^2 \sum\nl_1^n \vp_{k}.\end{equation} 
           \end{rem}
           
            It will be convenient to work with
            a general directed set $I$. 
            Given a family $(C_\alpha)$ of $C^*$-algebras, we denote by $\ell_\infty(I, \{C_\alpha\})$
            the $C^*$-algebra formed  of all the bounded families $(x_\a)$ (we view these as ``generalized sequences")
             with $x_\a\in C_\a$ for all $\a\in I$, equipped with the sup-norm.
            We denote by $c_0(I, \{C_\alpha\})$ the ideal formed of all $x=(x_\a)$
            such that  $\limsup\nl_\a\| x_\a\|=0$. As usual if $$Q: \ell_\infty(I, \{C_\alpha\})
            \to \ell_\infty(I, \{C_\alpha\})/c_0(I, \{C_\alpha\})$$ denotes the quotient map,
            then for any $x\in \ell_\infty(I, \{C_\alpha\})$ we have
            $$\|Q(x)\|= \limsup\nl_\a\| x_\a\|.$$

           The interest of $\vp$-morphisms is illustrated by the following simple lemma, that we will use only toward the end of this paper.
           
           \begin{lem}\label{ll5} Let $B,D$ be $C^*$-algebras. 
           Let $E\subset B$, $F \subset D$ be  f.d.s.a. subspaces. 
           For any $\d>0$ there is $\vp>0$ and 
             a f.d.s.a. superspace  $\cl E$ with  $E\subset \cl E\subset B$ 
           such that for any 
           $\vp$-morphism $\psi: \cl E \to C$ ($C$ any other $C^*$-algebra) we have
           $$\forall x\in E \otimes F\quad \|(\psi\otimes Id_D)(x)   \|_{C\otimes_{\max} D}\le (1+\d) \|x\|_{B\otimes_{\max} D}.$$
           \end{lem}
           \begin{proof} Let $I=\{ (\cl E,\vp)\}$ be the directed set of pairs with $E\subset \cl E$, $\vp>0$. Fix $\d>0$.
           It suffices to show  that there is $(\cl E,\vp)\in I$ such that 
           for all $C$, all
           $\vp$-morphisms $\psi: \cl E \to C$   and all   $x\in E \otimes F$ we have
           $ \|(\psi\otimes Id_D)(x)   \|_{C\otimes_{\max} D}\le (1+\d) \|x\|_{B\otimes_{\max} D}$.
           For each $\alpha\in I$ with $E\subset \cl E_\alpha$,
           there is a $C_\alpha$ and
           an $\vp_\alpha$-morphism $\psi_\alpha: \cl E_\alpha \to C_\alpha$,
           such that
           $$\|(\psi_\alpha \otimes Id_D)(x)\|_{C_\alpha \otimes_{\max} D} \ge (1+\d)^{-1} \sup \|(\psi  \otimes Id_D)(x)\|_{C \otimes_{\max} D}$$ 
           where the last supremum runs over all $C$ and all
           $\vp_\alpha$-morphism $\psi: \cl E_\alpha \to C $.
           Note that this last supremum is finite since the nuclear norm
           of each $\psi:\cl E_\alpha \to C $ is at most
           $(1+\vp_\a) \dim(\cl E_\alpha) $.
         Let  $\psi'_\alpha : B \to C_\alpha$ be the map that extends $\psi_\alpha$   by $0$ outside $\cl E_\alpha$ (we could use a linear map but this is not needed at this point).
         Consider $\psi'=(\psi'_\alpha): B \to  \ell_\infty (I; \{C_\alpha\})$ and 
         let $Q: \ell_\infty (I; \{C_\alpha\})\to \ell_\infty (I; \{C_\alpha\})/c_0(I; \{C_\alpha\})$ be the quotient map.
         Then $\pi=Q \psi': B \to \ell_\infty (I; \{C_\alpha\})/c_0(I; \{C_\alpha\})$ is clearly an isometric $*$-homomorphism.
         We have contractive morphisms
         $$\pi  \otimes Id_D : B\otimes_{\max} D \to [\ell_\infty (I; \{C_\alpha\})/c_0(I; \{C_\alpha\})]\otimes_{\max} D = [\ell_\infty (I; \{C_\alpha\})\otimes_{\max} D]/[c_0(I; \{C_\alpha\})\otimes_{\max} D]$$
         where the last $=$ holds by the ``exactness" of the max-tensor product (see e.g. \cite[p. 285]{P4}).
         Moreover, we have clearly a contractive morphism 
         $$  [\ell_\infty (I; \{C_\alpha\})\otimes_{\max} D]/[c_0(I; \{C_\alpha\})\otimes_{\max} D]
         \to  [\ell_\infty (I; \{C_\alpha \otimes_{\max} D\}) ]/[c_0(I; \{C_\alpha\otimes_{\max} D\}) ]. $$
         Since   $\pi(e)= Q ( (\psi_\alpha(e))_\a )$ for all $e\in E$ and since $x\in E \otimes D$, it follows that 
         $$\limsup\nl_\alpha  \|(\psi_\alpha \otimes Id_D)(x)\|_{C_\alpha \otimes_{\max} D} \le \|x\|_{B\otimes_{\max} D},$$
         which proves the desired result for each fixed given $x\in E \otimes D$.
         But since $E \otimes F$ is a finite dimensional subspace of ${B \otimes_{\max} D} $
         we may replace the unit ball by a finite $\d$-net in it 
         (or invoke Ascoli's theorem). We can deal with the latter case by
        enlarging $\cl E$ finitely many times. We obtain the announced result (possibly with $2\d$ instead of $\d$).\\
         A different proof can be obtained using the Blecher-Paulsen factorization, as   in \cite[Th. 26.8]{P4}.
             \end{proof}

 \begin{rem}\label{rr1}
           Consider a surjective $*$-homomorphism $q: C_1 \to B$ between $C^*$-algebras.
           It is well known that for any $\vp>0$ and any f.d. subspace
           $\cl E\subset B$ there is $\psi: \cl E \to C_1$ 
           with $\|\psi\|\le 1+\vp$ such that $q\psi(x)=x$ for any $x\in \cl E$
           (actually this even holds for $\vp=0$, see e.g. \cite[p. 46]{P4}).  \\     
           Thus,  if one replaces the c.b. norm by the usual one,
           the analogue of the LLP becomes universally valid, just like for local reflexivity (and the latter fact implies the former).
           \end{rem}
           However, we will need to work with quotient maps
           that admit a slightly stronger sort of lifting, as follows:
           \begin{dfn}\label{rd1} Let $q: C_1 \to  C_1/\cl I$ be a quotient $*$-homomorphism and let $B$ be a $C^*$-algebra.
           We say that  a $*$-homomorphism $\sigma: B \to C_1/\cl I$
       is ``almost multiplicatively locally liftable"  
        if    for any $\vp>0$ and any f.d.s.a. subspace
           $\cl E\subset B$ there is
           an $\vp$-morphism  $\psi: \cl E \to C_1$ 
             such that $\|q\psi(x)-\sigma (x)\| \le \vp\|x\|$ for any $x\in \cl E$.\\
           We will say that $q$ {almost allows liftings} 
           if  this holds for
           $B=C_1/\cl I$ and $\sigma=Id_B$ (or equivalently whenever $\sigma$ is an isomorphism). In that case, any $*$-homomorphism $\sigma: B \to C_1/\cl I$
       is almost multiplicatively locally liftable. 
           \end{dfn}
           
             Let  $q: C_1 \to C_1/\cl I$ be a surjective $*$-homomorphism.
               Let $I$ be any directed set.
               We denote by $q^\sharp:   \ell_\infty(I;C_1)/ c_0(I;C_1)\to \ell_\infty(I;C_1/\cl I)/ c_0(I;C_1/\cl I)$   the $*$-homomorphism
           associated to $q$ acting an each coordinate, and by
           $\nu: C_1/\cl I\to  \ell_\infty(I;C_1/\cl I)/ c_0(I;C_1/\cl I)$   the embedding
           associated to $Id_{C_1/\cl I}$ acting an each coordinate, or 
           equivalently   the map that takes $x\in C_1/\cl I$ to the 
           equivalence class of the function constantly equal to $x$ on $I$.
               
            \begin{pro}\label{pus1} The following properties of a   $*$-homomorphism $\sigma: B \to C_1/\cl I$ are equivalent.
            \item[{\rm (i)}] The map $\sigma$ is almost multiplicatively locally liftable.
                 \item[{\rm (ii)}] There is a directed set $I$ 
                 and a $*$-homomorphism 
                  $\pi:   B \to \ell_\infty(I;C_1)/ c_0(I;C_1)  $   (which will automatically be an embedding)
                  such that 
                  $q^\sharp \pi=\nu \sigma$.           
            \end{pro}
            $$\xymatrix{&B\ar[d]^{\sigma} \ar[rr]^{\pi \qquad} && \ell_\infty(I;C_1)/ c_0(I;C_1) \ar[d]^{q^\sharp}  \\
&  {C_1/\cl I} \ar[rr]^{ \nu  \qquad}&& \ell_\infty(I;{C_1/\cl I})/ c_0(I;{C_1/\cl I})}$$
\begin{proof} Assume (i). Let $I$ be the directed set formed of all
           pairs $\alpha=(\cl E,\vp)$ where $\cl E$ is a  f.d.s.a. subspace of $B$ and $\vp>0$,
           equipped with the usual ordering so that
           $\alpha\to \infty$ in the corresponding net means that $\cl E\to B$ and $\vp\to 0$.
           Let       $\alpha=(\cl E,\vp)$ be such a pair. 
Let $\psi: \cl E \to C_1$
             be an $\vp$-morphism such that $\|q\psi(x)-\sigma x\| \le \vp\|x\|$ for any $x\in \cl E$.
           Then we define a linear map $\psi_\alpha: B \to C_1$   by setting
            $\psi_\alpha(x)=\psi(x)$ if $x\in \cl E$
           and (say) $\psi_\alpha(y)=0$ whenever $y$ belongs to a complementary (to $\cl E$) subspace
           that we can choose arbitrarily.    
           We denote by
            \begin{equation}  \label{ee4}\pi: B \to \ell_\infty(I;C_1)/ c_0(I;C_1),\end{equation} 
            the mapping 
           that takes $x\in B$ to  $(\psi_\alpha(x))$ modulo $c_0(I;C_1)$.
            It is easy to check that $\pi$ is
           an isometric  $*$-homomorphism   such that 
                  $q^\sharp \pi=\nu\sigma$.\\
                  Conversely, assume (ii). 
                  Let $Q: \ell_\infty(I;C_1) \to \ell_\infty(I;C_1)/ c_0(I;C_1)$
                  be the quotient map. Let $\cl E\subset B$ be a f.d.s.a. subspace.
                  We set $\hat {\cl E}= \cl E + {\rm span}[\cl E\cl E]$.
                  By Remark \ref{rr1},
                there is $\psi: \hat{\cl E} \to \ell_\infty(I;C_1)$ 
           with $\|\psi\|\le 1 $ such that $Q\psi(x)=\pi(x)$ for any $x\in \hat{\cl E}$.
           Replacing $\psi $ by $(\psi+\psi_*)/2$ we may assume $\psi$ s.a.
           Let $\psi_\a: \hat{\cl E} \to C_1$ be the coordinates of $\psi$ so that $\psi=(\psi_\a)$. Note $q^\sharp Q\psi(x)=q^\sharp \pi(x)=\nu\sigma (x)$ for any $x\in \hat{\cl E}$. Equivalently, if we set $\nu_\a(x)=x$ for any $x\in C_1/\cl I$, we have
           $(q \psi_\a(x))_\a -(\nu_\a\sigma (x))_\a \in c_0(I;C_1/\cl I)$, which means
           $\limsup_\a\| q \psi_\a(x)-\sigma (x)\|=0$ for any $x\in \hat{\cl E}$.
           A fortiori this holds for any $x\in  \cl E$.
           Now, since   $\cl E\cl E\subset \hat{\cl E}$, if $x,y \in \cl E$ we have
           $Q (\psi(xy)-\psi(x)\psi(y))= \pi(xy)-\pi(x)\pi(y) =0$, which means
           $\limsup_\a\| \psi_\a(xy)-\psi_\a(x)\psi_\a(y)\|=0$.
           Thus choosing $\a$ large enough (and invoking the compactness
           of the unit ball of $\hat{\cl E} $) we find 
           $\psi=\psi_\a$ satisfying the properties required to check (i).
\end{proof}

           \begin{rem}\label{rr2}
           It is   known that   a general $q$ does not almost allow liftings.
           Indeed, 
         if $q$ is the surjection from $\B$ to the Calkin algebra
         and if $S\in \B$ is the shift then the subspace $\cl E$ spanned
         by $\{q(1),q(S),q(S^*)\}$ does not satisfy the local lifting
         described in Definition \ref{rd1}. This can be checked
         using the Fredholm index.
         According to D. Voiculescu who  kindly explained it to me, this kind of example    was known around the time of the Brown-Douglas-Filmore theorem.
         
           However, it is   known that    the associated
           surjective $*$-homomorphism $\hat q: C_0(C_1) \to C_0({C_1/\cl I})$
           (extended to the cone algebras) almost allows liftings.
           This is the content of  Lemma \ref{ll7} below.
           \\ For a general $q$, Remark  \ref{rr1} gives us only a linear 
           isometric embedding $\pi$.           
           \end{rem} 
           \section{The cone algebra}

Let  $C_0=C((0,1])$  and  $C=C([0,1])$.  For  any  $C^*$-algebra  $A$,
  we  denote  by  $C_0(A)=  C_0  \otimes_{\min}  A$  the  so-called  cone  algebra
  of  $A$.  
    When  dealing  with  a  mapping  $u:  A  \to  B$  between  $C^*$-algebras
  (or  operator  spaces)  we  will  denote  by  $u_0:  C_0(A)  \to  C_0(B)$
  the  map  extending    $Id_{C_0}  \otimes  u$.
We  denote  by  $\ell_\infty(A)  $
the  $C^*$-algebra  formed  of  all  bounded  sequences  of  elements  of  $A$,
and  by  $c_0(A)  \subset  \ell_\infty(A)$  the  ideal  formed  by  
the  sequences  that  tend  to  $0$.

Let $q: C \to B$ be a surjective $*$-homomorphism and let $\I=\ker(q)$.
Let  $(\sigma_n)$  be  a  quasicentral  approximate  unit  in  $\I$.
Our  construction  would  be  much  simpler  of  we  could
find  $(\sigma_n)$   (and hence $(1-\sigma_n)$)  formed  of  projections.
Then  the  mappings  $x\mapsto  (1-\sigma_n ) x$  would  be  approximatively  multiplicative.
The next lemma  somehow  produces  a  way  to  go  around  that  difficulty  by  passing  to  the  cone  algebras. 
It is closely   related  to    Kirchberg's  \cite[\S  5]{Kir},
but    we  learnt it from  \cite[Lemma  13.4.4]{BO}.  A similar idea already
appears in \cite[Lemma 10]{CoH} for the suspension algebra in the context of
approximatively  multiplicative families indexed by a continuous parameter in $(0,\infty)$.

            \begin{lem}\label{ll7}
            Let $q: C \to B$ be a surjective $*$-homomorphism between $C^*$-algebras.
            Let $q_0: C_0(C) \to C_0(B)$ be the associated one on $C_0(C)=C_0\otimes_{\min} C$.
            Then $q_0$   {almost allows liftings}.
            \end{lem}
            \begin{proof} Let $\cl I=\ker(q)$.
            Let $(\sigma_\a)$ be a net forming a quasicentral approximate unit of $\cl I$, indexed by a directed set $I$. 
            This means $\sigma_\a\ge0$, $\|\sigma_\a\| \le 1$,
            $\|\sigma_\a x-x\|\to 0$ for any $x\in \cl I$ and
            $\|\sigma_\a c-c\sigma_\a\|\to 0$ for any $c\in C$.
            We identify $C_0(C)=C_0\otimes_{\min} C$ with the set of $C$ valued functions
            $f: [0,1] \to C$ such that $f(0)=0$. 
            Note that
            $$ C_0(C)/C_0(\cl I)= C_0( C/\cl I)=C_0(B).$$
            The set of polynomials
            $\cl P_0={\rm span}[t^n\mid n>0]$ is dense in $C_0$.
            Let $\rho_\a: C_0 \otimes C \to C_0(C)$ be the map taking 
            $t\mapsto f(t) c$ ($f\in C_0,c\in C$) to
            $t\mapsto f(t(1-\sigma_\a)) c$. For instance (monomials)
            $\rho_\a$ takes $t\mapsto t^n c$ to
            $t\mapsto t^n (1-\sigma_\a)^n c$.
            Note $\|1-\sigma_\a\|\le 1$ therefore for any $f\in C_0$ the function
            $t\mapsto f(t(1-\sigma_\a))$ is in $C_0(C)$
            with norm $\le \|f\|_{C_0 }$. This shows that
            $\sup_\a\|\rho_\a(y)\|<\infty$ for any $y\in C_0\otimes C$, so that $(\rho_\a)$ defines
            a map $\rho: C_0 \otimes C \to \ell_\infty(C_0(C))$.
           Let $$  \cl L(C_0(C))=\ell_\infty(C_0(C))/  c_0(C_0(C)),$$
           and similarly for $C_0(B)$. Let $Q_C: \ell_\infty(C_0(C)) \to \cl L(C_0(C))$ be the quotient map and similarly for $Q_B$.
             Since $\sigma_\a\in \cl I$
             we have for all $y\in \cl P_0\otimes C$ (and hence all $y\in \cl C_0\otimes C$)
              \begin{equation}\label{us8} \limsup\nl_\a \| q_0\rho_\a (y) - q_0(y)\| =0. \end{equation}
            Since $\sigma_\a$ is quasicentral we have
            $$\forall x,y\in C_0\otimes C\quad 
            \limsup\nl_\a\|
            \rho_\a(xy)-\rho_\a(x)\rho_\a(y)\|=0,$$
            indeed this reduces to the case of monomials which is obvious,
            and also $$\limsup\nl_\a\|\rho_\a(x)^*-\rho_\a(x^*)\|=0.$$
            It follows that after composing $\rho$ by the quotient map
            $Q_C: \ell_\infty(C_0(C)) \to \cl L(C_0(C))$ we obtain
            a  map 
            $\hat\pi : C_0 \otimes C \to \cl L(C_0(C))$ which
            is a $*$-homomorphism, such that $\hat\pi=Q_C\rho$ on $C_0\otimes C$.
            Since $C_0$ is nuclear, this  extends to a $*$-homomorphism
            (still denoted by $\hat\pi$)  defined on the whole of 
            $C_0(C)$, whence
             $\hat\pi: C_0(C) \to \cl L(C_0(C)) .$
     We have
             \begin{equation}\label{us9}\forall f \in C_0\otimes \cl I\quad \limsup\nl_\a\|
            \rho_\a(f)\|=0 ,\end{equation}
            and hence  $\hat\pi( C_0\otimes \I)=Q_C\rho ( C_0\otimes \I)=0$.
            Therefore,  after passing to the quotient by $C_0\otimes \I \subset \ker(\hat\pi)$, we derive from $\hat\pi$
            a $*$-homomorphism $$\pi: C_0(B) \to  {\cl L}(C_0(C)) ,$$
            such that $\hat \pi= \pi q_0$.
            Let    $q^\sharp: {\cl L}(C_0(C)) \to {\cl L}(C_0(B))$ denote again
            the quotient map associated to $q_0$.
            We claim that  for any $b=q_0(y)\in C_0 \otimes B $ (with   $y\in C_0\otimes C)$)   we have           $q^\sharp \pi(b)= 
             \nu(b)$
            where $\nu(b)$ is as earlier the element
            of ${\cl L}(C_0(B))$ that is the equivalent class
            of the family $(b_\a)$ defined by $b_\a=b$ for all $\a$.
            Indeed, 
            denoting simply by $Id  \otimes q_0: \ell_\infty(C_0(C)) \to \ell_\infty(C_0(B))$ the coordinate wise extension of $q_0$, we have
              $q^\sharp Q_C=Q_B (Id  \otimes q_0)$.
              Moreover,  \eqref{us8} implies
            that
           $ Q_B (Id  \otimes q_0) \rho(y) =Q_B([q_0(y)])$
           where $[q_0(y)]$ denotes the constant family equal to $q_0(y)$,
           so that $Q_B([q_0(y)])=\nu(q_0(y))$.
               Thus         $q^\sharp \pi(b)= 
            q^\sharp \pi q_0(y)= q^\sharp \hat\pi(y)=
             q^\sharp Q_C\rho(y)
             =Q_B (Id  \otimes q_0) \rho(y)
            =\nu(q_0(y))=\nu(b)$.
           This completes the proof by Proposition \ref{pus1}.
             \end{proof}

 \begin{lem}\label{l5}
           Let $C,C_1,B$ be  $C^*$-algebras. We assume given 
           an injective (and hence isometric) $*$-homomorphism
           $i: C \to B$ and a surjective one $q: C_1 \to B$.  
           Thus
           we have $B=C_1/\cl I$ with $\cl I=\ker(q)$.
           Assume that $C$ has the LLP and that $q$ {almost allows liftings}.           Then for any pair of f.d.s.a. spaces $E\subset C$ and $\cl E\subset B$
           such that $i(E)\subset \cl E$ and for any $\vp>0$ there is
           an $\vp$-morphism ${\psi}: \cl E \to C_1$ such that 
           $\|\psi {i}_{|E}\|_{cb}\le 1+\vp$ that approximately   lifts $q$  on $\cl E$ in the sense
           that $\|q{\psi}(x)-x\|\le \vp \|x\|$ for any $x\in \cl E$.
           
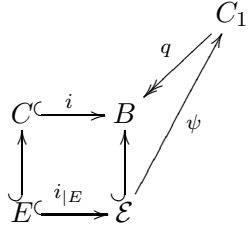
           $$\xymatrix{ &&&{C_1}\ar@{>>}[dl]_{ q}  \\
&   
  C\ar@{^{(}->}[r]^{  \  \  i  \quad  }  &  B&\\
  & E\ar@{^{(}->}[u]\ar@{^{(}->}[r]^{  \  \  i_{|E}  \quad  }  &  \cl E\ar@{^{(}->}[u]\ar [uur]_{{\psi}} &}$$
           \end{lem}
          \begin{proof} 
          Let $(\psi_\alpha)$ and $\pi$ be the maps in \eqref{ee4} for the case $\sigma=Id_B$. By the definition of $\pi$,  we have
          $\|q\psi_\a(x)-x\| \le \vp_\a\|x\|$ for any $x\in \cl E_\a$ (here $\a=(\cl E_\a,\vp_\a)$).
          Let $\chi: \ell_\infty(I;C_1) \to \ell_\infty(I;C_1)/ c_0(I;C_1)$ denote the quotient map.
          By the LLP of $C$ we can lift  $\pi {i}_{|E}$ : this gives us 
           a map $u=(u_\a): E \to \ell_\infty(C_1)$ with $\|u\|_{cb}=\sup_\alpha\|u_\alpha\|_{cb}\le 1$
          such that $\chi u(x)=\pi i (x)$ for all $x\in E$.
          This means that
          $$\forall x\in E\quad \limsup\nl_\alpha \|u_\alpha(x) -\psi_\alpha i(x) \|=0.$$ 
          Since $\dim(E)<\infty$, it follows that  
          $\limsup\nl_\alpha \|u_\alpha  -\psi_\alpha{i}_{|E}  \|_{cb}=0$ and hence
            $\limsup\nl_\alpha \|  \psi_\alpha {i} _{|E}  \|_{cb}\le 1$. 
            Thus the lemma follows by taking $\psi=\psi_\alpha$ for $\alpha$ sufficiently
            ``large" in the net.
          \end{proof}
 \section{Main  construction}
 
 Recall      $\mathscr{C}  =C^*(\F_\infty)$
and  $\mathscr{B}  =B(\ell_2)$.
By the universality of $\B$,  there is an embedding $\C\subset \B$.
It  is  known  (see  \cite[Lemma  2.4]{Kir})  that  any  separable  $C^*$-algebra
embeds  in  a  separable  one  with  the  WEP.  
Thus there is a  separable  $C^*$-algebra $B$
  with  the  WEP  such that   $\mathscr{C}\subset B$.
  Let
  $i:  \mathscr{C}  \to  B$  be  an  embedding  (i.e.  a  faithful  $*$-homomorphism).
  Let  $q:  \mathscr{C}  \to  B$  be  a  surjective  $*$-homomorphism.
  The  relevant  diagrams  are  as  follows:
  
  $$\xymatrix{&&\mathscr{C}\ar[d]^{  q  }  \\
&  \mathscr{C}  
  \ar@{^{(}->}[r]^{  \  \  i  \quad  }  &  B} \xymatrix{&&C_0(\mathscr{C})\ar[d]^{q_0}  \\
&  C_0(\mathscr{C})  
  \ar@{^{(}->}[r]^{  \  \  i_0  \quad  }  &  C_0(B)}$$
    
We will now tackle the lifting problem expressed by these diagrams:
we will locally lift $Id_{C_0(B)}$ by an approximatively multiplicative map using   Lemma \ref{l5}.
To  shorten  the  notation  we  set
$$  \cl L=\ell_\infty(C_0(\C))/  c_0(C_0(\C)).$$
We  denote  by  $Q:  \ell_\infty(C_0(\C))  \to  \cl L$  the  quotient  map
so  that,  as  before,  we  have  
$$\forall  x=(x_n)  \in  \ell_\infty(C_0(\C))\quad    \|Q(  (x_n)  )\|_{ \cl L}=\limsup  \|x_n\|_{C_0(\C)}.$$
 
The next lemma (for a single $u: S \to E$) is the basic step.
\begin{lem}\label{l2}  Let  $F,  E\subset  C_0(\C)$  be      f.d.s.a.  subspaces
such  that  $F.F  \subset  E$.
Fix  $n$  and  $\vp>0$.
There  is  a  f.d.s.a.  subspace  $E_1\subset  C_0(\C)$
and  a  s.a.  map  $T:  E  \to  E_1$  such  that
\item{\rm  (i)}  For  any  subspace  $S\subset  \ell_1^n$  and  any  $u$ in   ${CB(S,E)}$  there  is  $\tilde  u:  \ell_1^n  \to  E_1$  such  that
$$ \tilde  u_{|S}  =  Tu \text{    and    } 
\|\tilde  u\|_{cb}\le    (1+\vp)\|u\|_{cb}.$$
\item{\rm  (ii)}  $\|T\|_{cb}\le  1+\vp$  and  $\|T^{-1}_{|T(E)}\|_{cb}\le  1+\vp$.
\item{\rm  (iii)}  $T(F)T(F) \subset E_1$ and for  all  $x,y\in  F$  we  have
$\|T(xy)-T(x)T(y)\|\le  \vp\|x\|  \|y\|.$
\end{lem}
\begin{proof}  
We  first  show  that  to  check    (i)  for  a  fixed  $\vp>0$,  it  suffices  to  check  it
with  $\vp$  replaced
by  (say)  $\vp/4$
for  a  suitably  chosen  \emph{finite}  set  of  subspaces  $S\subset  \ell_1^n$
and  a  suitable  \emph{finite}  set  of  $u$'s  in  the  unit  ball  of  $CB(S,E)$.
Since  $n$  is  fixed  the  set  of  $k$-dimensional  subspaces
$S\subset  \ell_1^n$  can  be  viewed  as  being  compact  for  the
Hausdorff  distance,  so  that  it  admits  a  finite  $\d$-net  for  any  $\d>0$.
In  other  words,        by    perturbation,
to  obtain  (i)  for  a  given  $\vp>0$,    it  suffices  to  check  (i)  with  $\vp$  replaced
by  $\vp/2$
for  a  suitably  chosen  \emph{finite}  set  of  subspaces  $S\subset  \ell_1^n$.
Then  $S$  and  $\vp>0$  being  fixed,  since  the  unit  ball  of  $CB(S,E)$  is  also  compact,
  to  show  (i)  with  $\vp$  replaced
by  $\vp/2$  for  all  $u$  in  the  unit  ball  of  $  CB(S,E)$  it  suffices  to  show  (i)  with  $\vp$  replaced
by  $\vp/4$  
  for  only  a  suitably  chosen    \emph{finite}  set  of  $u$'s  in  it.
\\
Being now left with a finite set $(u_i)_{1\le i\le N}$ of $u$'s to extend, we may observe that
we may handle these simply one after the other:
having handled the first $u_1$ 
by producing $\tilde {u_1}$ and $T_1: E \to E_1$
we 
replace $F$ by $T_1(F)$, $u_2: S \to E$ by $T_1 u_2: S \to E_1$
and we apply the same procedure to it, and so on.
Since the number $N$ of $u$'s is fixed,
and $N\vp$ can be chosen arbitrary small,
after $N$ steps, we will have $\|T_1\cdots T_N\|\le (1+\vp)^N$
and the same for the inverses, and
  taking  \eqref{ee3} into account to tackle (iii),
we obtain the announced result. 
\\Thus it suffices to prove Lemma \ref{l2} for a fixed $S$ and a single $u:S \to E$.
By homogeneity we may assume $\|u\|_{cb}=1$.
   Since  $B$  and  hence  $C_0(B)$  has  the  WEP,    the  map
$i_0  u:  S  \to  C_0(B)$
admits  an  extension  $v  :  \ell^n_1\to  C_0(B)$
with  $\|v\|_{cb}  \le(1+\vp)  \|i_0  u\|_{cb}=(1+\vp)\|u\|_{cb}=1+\vp.$  We  have  $v_{|S}  =i_0u$. Let us choose a f.d.s.a.  subspace $\cl E \subset C_0(B)$ large enough so that  $i_0(E)+ v(\ell^n_1) \subset \cl E $.
By Lemma \ref{l5}, for any $0<\vp'<\vp$
there is an $\vp'$-morphism $\psi: \cl E \to C_0(\C)$
with  $\|\psi {i_0}_{|E}\|_{cb}\le 1+\vp'$ such
that $\|q_0{\psi}(x)-x\|\le \vp' \|x\|$ for any $x\in \cl E$. We then set  
$$
E_1=\psi (\cl E) + \psi {i_0}(F)\psi {i_0}(F) \text{  and  }
T=\psi {i_0}_{|E} : E \to E_1.$$ Note $T(F)T(F) \subset E_1$. 
Moreover, we define $\tilde u : \ell^n_1\to E_1$ by
$\tilde u(z)= \psi v(z)$ for all $z\in \ell^n_1$. Then 
$ \tilde  u_{|S}  =  Tu$, $\|\tilde u\|_{cb}\le (1+\vp') (1+\vp)$ and $\|T\|_{cb}\le 1+\vp'$. Since $\vp'$ can be chosen arbitrarily small, we may  assume $\vp'<1/\dim(\cl E)$.
Let $\kappa: \cl E \to C_0(B)$ denote the inclusion map.
Since $\|q_0{\psi} -\kappa\|\le \vp' $    we have (see e.g. \cite[p. 75]{P4})
$\|q_0{\psi} -\kappa \|_{cb}\le \vp'  \dim(\cl E) <1 $. 
Similarly, $\| q_0 T   -\kappa {i_0}_{|E}\|_{cb}\le \dim( E)\| q_0 T   -\kappa {i_0}_{|E}\|\le \vp'  \dim( E) <1 $, and hence
$\|   T (e)\|\ge \| q_0 T (e)\| \ge \|\kappa{i_0}(e)\|-\vp'\dim(  E) \|e\|=(1-\vp'\dim(  E))\|e\|$ for any $e\in E$.
This shows $\|T^{-1}_{|T(E)}\| \le  (1-\dim(  E)\vp')^{-1}$. A similar reasoning
shows that $\|T^{-1}_{|T(E)}\|_{cb} \le  (1-\dim(  E)\vp')^{-1}$
(a variant of this reasoning actually shows $\|\psi^{-1}_{|\psi(\cl E)}\|_{cb} \le  (1-\dim(\cl E)\vp')^{-1}$).
\\
It only remains
to adjust $\vp$ and $\vp'$ to match (i) and (ii) as stated above.
Lastly,  (iii) follows since $\psi$ is an $\vp'$-morphism
on $\cl E$ and $i_0$ an isometric $*$-homomorphism. 
\end{proof}

\begin{thm}\label{t1}
Let  $(Z_n)$  be  a  sequence  of  finite  dimensional  self-adjoint  (f.d.s.a.)  (operator)  subspaces  of  $C_0(\C)$.
There  is  a  non-nuclear separable $C^*$-algebra  $A$  with  both  the  WEP  and  the  LLP.  Moreover,
    for  any  $n$  and  any  $\vp>0$  there  is  a  subspace  $Z\subset  A$  such  that
$d_{cb}(Z_n,Z)<1+\vp$.
\end{thm}
\begin{rem} Since the WEP and the LLP pass to the unitization
the unitization of $A$ has the properties in Theorem \ref{t1}.
Moreover, by a result due to Voiculescu (see e.g. \cite[p. 251]{BO}) for any separable $C^*$-algebra
    $A$, the cone $C_0(A)$ is quasidiagonal.
    Thus replacing $A$ by the unitization of $C_0(A)$ we obtain a unital quasidiagonal example
    as in Theorem \ref{t1}.
\end{rem}

\begin{rem}
It is easy to deduce from \eqref{us10}  that for 
 any nuclear $C^*$-algebra $C$, the algebra 
 $\C$ is locally equivalent 
 to $C \otimes_{\min} \C$. In particular 
 $\C$ is locally equivalent 
 to $C_0(\C)$.
\end{rem}
\begin{rem}  Since  the  set  of  f.d.    subspaces    of    $C_0(\C)$  is  $d_{cb}$-separable
if  we  choose  for  $(Z_n)$  a  dense  sequence,  then
for  any  f.d.  subspace  $Z'$ in $  C_0(\C)$ (or in $\C$)  and  any  $\vp>0$  there  is  a  subspace  $Z\subset  A$  such  that
$d_{cb}(Z',Z)<1+\vp$.  In  the  converse  direction, by Remark \ref{rus3} the  LLP
of  $A$  implies  that    $d_{S\C}(A)=1$.
Thus,  with this choice of $(Z_n)$,  $A$  and  $\C$  are   locally equivalent.  
\end{rem}

\begin{lem}\label{l3}  Let  $(Z_n)$  be  as  in  Theorem  \ref{t1}.  Let  $\vp_n>0$  be  such  that  $\sum  \vp_n<\infty$.
There  is  a  sequence  of  f.d.s.a.  subspaces  $E_n  \subset  C_0(\C)$  
and  s.a.  maps  $T_n  :  E_n  \to  E_{n+1}$      such  that  we  have  for  any  $n\ge  1$
\item{\rm  (i)}  $\forall  S\subset  \ell_1^n,\forall  u:  S  \to  E_n$
$\exists  \tilde  u  :  \ell_1^n  \to  E_{n+1}$  such  that
$$  \tilde  u  _{|S}  =T_n  u      \text{    and    }  \|\tilde  u\|  \le  (1+\vp_n)    \|u\|_{cb}.$$
\item{\rm  (ii)}  $\|T_n\|_{cb}\le  1+\vp_n$  and  $\|{T_n^{-1}}_{|T_n(E_n)}\|_{cb}\le  1+\vp_n$.
\item{\rm  (iii)}  For  any  $n\ge  2$  we  have      $T_{n-1}(E_{n-1})T_{n-1}(E_{n-1})  \subset  E_{n}$  and
    \begin{equation}\label{e10}
    \forall  x,y\in  T_{n-1}(E_{n-1})\quad  \|T_n(x)T_n(y)-T_n(xy)\|\le  \vp_n  \|x\|\|y\|.\end{equation}  
  \item{\rm  (iv)}  For  any  $n\ge  2$  we  have  $Z_n\subset  E_n$.
\end{lem}

\begin{proof}  We  construct  $E_n,T_n$  by  induction  on  $n$  starting  from  an  arbitrary
$E_1=Z_1$.  At  the  initial  step  $n=1$  (i)  is  trivial    and  (iii)  is  void  so  that  we  simply  may  set
$E_2=E_1+Z_2+{\rm  span}[E_1.E_1]$  and  let  $T_1:  E_1  \to  E_2$  be  the  natural  inclusion.
We    have  the  required  properties  with  $\vp_1=0$.
\\
Assume  that  $(E_k)_{k\le  n}$  and  $(T_k)_{k<n}$    have  been  constructed  satisfying  (i)  (ii)  (iii)  (iv).
For  the  induction  step  we  must  produce  $E_{n+1}$  and  $T_n$.
We  find  $T_{n}:  E_n  \to  E_{n+1}$  using  Lemma  \ref{l2}  applied  to  $E=E_n$ 
with    $\vp=\vp_n$,
and  taking  $F=T_{n-1}(E_{n-1})$.  This  gives  us  $E_{n+1}$  and  $T_{n}:  E_n  \to  E_{n+1}$  
(equal  to  the  $T$  given  by    Lemma  \ref{l2})  satisfying  (i)  (ii)  
and  \eqref{e10}.
But  since  (i)  (ii)      remain  unchanged  if  we  enlarge  $E_{n+1}$,
we  may  replace  our  subspace  by  $E_{n+1}+Z_{n+1}+{\rm  span}[T_n(E_n)T_n(E_n)]$      to  ensure  that
$T_n(E_n)T_n(E_n)\subset  E_{n+1}$  and  $Z_{n+1}\subset  E_{n+1}$,  so  that
  (iv)  also  holds  at  the  next  step.
This  completes  the  proof  by  induction.
  \end{proof}
  
\begin{proof}[Proof  of  Theorem  \ref{t1}]
Let  $\vp_n>0$  be  such  that  $\sum  \vp_j  <\infty$.  Let  $\d_n=\sum_{j>n}  \vp_j$.  Note  that  the  infinite  product
$\prod_{j\ge  1}  (1+\vp_j)    $  converges.  
We  define  $\eta_n>0$  by  the  equality
\begin{equation}\label{e12}
1+  \eta_n=  \prod\nl_{j\ge  n}  (1+\vp_j)    ,\end{equation}
so  that  $\eta_n  \to  0$.  
\\
  Let  $(E_n)$  be  as  in  Lemma  \ref{l3}.
We  will  work  in  the  ambient  $C^*$-algebra  $  {\cl L}=\ell_\infty(C_0(\C))/  c_0(C_0(\C))$,
with  quotient  map  $Q:  \ell_\infty(C_0(\C))\to  {\cl L}$.
We  denote  by  $L\subset  \ell_\infty(C_0(\C))$  the  subspace
formed  of  the  sequences  $(x_n)$  such  that  $x_n\in  E_n$  for  all  $n$,
so  that  $L  \simeq  \ell_\infty(\{E_n\})$.

We  introduce
\def\t{\theta}
a  mapping  $\t_n  :  E_n  \to  \ell_\infty(C_0(\C))$,  with  values  in  $L$,  
defined  by
$$  \forall  x\in  E_n\quad  \t_n(x)=(0,\cdots,  0,x,T_n(x),T_{n+1}T_n(x),T_{n+2}T_{n+1}T_n(x),\cdots)$$
where  $x$  stands  at  the  $n$-th  place.\\
By  (ii)  in  Lemma  \ref{l3},  we  have  
$$\forall  n\ge  1  \quad  \|    \t_n\|_{cb}\le  1+\eta_n.$$
We  define
  the  subspace
$Y_n\subset  {\cl L}$    
by  setting
$Y_n=  Q\t_n(E_n).$
Then  $Y_n$  is  a  f.d.s.a.  subspace  of  ${\cl L}$  such  that
$Y_n\subset  Y_{n+1}$
for  all  $n\ge  1$.  Indeed,
we  have
$$\forall  x\in  E_n\quad      Q\t_n(x)=Q\t_{n+1}(T_n(x)    ),$$
because  $\t_n(x)  -\t_{n+1}(T_n(x)    )$  is  a  finitely  supported  element  of  $\ell_\infty(C_0(\C))$.
Let
$$A=\ovl{\cup  Y_n}  \subset  {\cl L}.$$
A  priori  this  is  a  s.a.  subspace.  We  will  see  below  that  it  is  actually  a  $C^*$-subalgebra.

We  will  first  show  that  $d_{cb}(  Y_n,E_n)  \le  \prod_{j\ge  n}  (1+\vp_j)=1+\eta_n$.
Note  first  of  all  that  the  map  $w_n=  Q  \t_n:  E_n  \to  Y_n$
satisfies    $\|w_n\|_{cb}  \le  \|\t_n\|_{cb}  \le1+\eta_n$.
  By  definition  of  $Y_n$  we  know  $Y_n=w_n(E_n)$.  
  We  claim  that  \begin{equation}  \label{e11}  \|  {w_n^{-1}}_{|Y_n}  \|_{cb}  \le  \prod_{j\ge  n}  (1+\vp_j).  \end{equation}
  Indeed,  for  any  $x\in  E_n$
  we  have
  $\|w_n(x)\|=\limsup_k  \|T_{n+k}\cdots  T_{n+1}T_n(x)\|$  and  hence  using
  (ii)  in  Lemma  \ref{l3}
    and  one  more  telescoping  argument
  we  find
  $\|w_n(x)\|  \ge  \prod_{  j\ge  n}  (1+\vp_j)^{-1}  \|x\|$.
  This  shows  
  $  \|  {w_n^{-1}}_{|Y_n}  \|  \le  \prod_{j\ge  n}  (1+\vp_j)  $.
  A  simple  modification  gives  us  the  same  for  the  cb-norm,  whence  the  claim.\\
  Now  we  have  a  commuting  diagram
 factorizing    the  inclusion  $Y_n  \subset  Y_{n+1}$,
  as  follows:
  $$\xymatrix{&E_n\ar@{}[dr]\ar[r]^{T_n\  \  }  &E_{n+1}\ar[d]^{w_{n+1}}    \\
&  Y_n  \ar[u]^{w^{-1}_{n}}
  \ar@{^{(}->}[r]    &  Y_{n+1}}$$
Recall  $\d_n=\sum_{j>n}  \vp_j$.  
Let $c= \prod_1^\infty (1+\vp_j)$.To  check  that  $A$  is  a  subalgebra  of  ${\cl L}$
we  will  show  that
\begin{equation}  \label{e15}
\forall  a,b\in  Y_n\quad  d(ab,  Y_{n+1})\le  (2c^3+c^2)  \d_n  (1+\eta_n)^2\|a\| \|b\|.\end{equation}
Assume  $a=w_n(x)=w_{n+1}(T_n(x)    )$  and  $b=w_n(y)=w_{n+1}(T_n(y)    )$.
Then
$$ab=  Q\t_{n+1}(T_n(x)    )  Q\t_{n+1}(T_n(y)    )=Q(\t_{n+1}(T_n(x)    )\t_{n+1}(T_n(y)    )$$
$$=
Q((0,\cdots,  0,T_n(x)T_n(y),T_{n+1}T_n(x)T_{n+1}T_n(y),T_{n+2}T_{n+1}T_n(x)T_{n+2}T_{n+1}T_n(y),\cdots)),$$
where  $T_n(x)T_n(y)\in  E_{n+1}$  stands  at  the  $n+1$-place.
We  will  compare  $ab$
with  $Q\t_{n+1}  (T_n(x)T_n(y))$.
To lighten the notation, we set 
$$T_{n+k,n+1}=T_{n+k}\cdots  T_{n+1}.$$
We  have
$$\|  ab-  Q\t_{n+1}  (T_n(x)T_n(y))\|=\ovl\lim_{k}      \|T_{n+k,n+1}T_n(x)T_{n+k,n+1}T_n(y)-  T_{n+k,n+1}[T_n(x)T_n(y)]\|.$$
 We  claim
that
 \begin{equation}  \label{us6} 
 \|T_{n+k,n+1}T_n(x)T_{n+k,n+1}T_n(y)-  T_{n+k,n+1} [T_n(x)T_n(y)]\|\le
(2c^3+c^2)
(\sum\nl_{n<j\le  n+k}  \vp_j)\|x\|\|y\|.\end{equation} 
By homogeneity we may assume $\|x\|=\|y\|=1$.
For convenience observe $\| T_{n+k,n+1}T_n\|\le c$  and $\|T_n\|\le c$,
whence the following trivial a priori bound for the left hand side of \eqref{us6}
 \begin{equation}  \label{us5} 
 \|T_{n+k,n+1}T_n(x)T_{n+k,n+1}T_n(y)-  T_{n+k,n+1} [T_n(x)T_n(y)]\|\le
2c^3 .\end{equation} 
Let us check \eqref{us6} by induction on $k$.
The case $k=1$ is clear by  \eqref{e10}. 
  Assume \eqref{us6} proved for a given $k$
  and let us deduce the same for $k+1$.
  We  will  use     an  easy
  telescoping  sum  argument. 
  Since  
    $\|T_{n+k+1}\|\le  1+\vp_{n+k+1}$, 
    taking \eqref{us5} into account, we have by \eqref{us6} 
    $$\|T_{n+k+1}   [T_{n+k,n+1} T_n(x)T_{n+k,n+1}T_n(y)]-  T_{n+k+1}T_{n+k,n+1}[T_n(x)   T_n(y)]\|
    $$
    $$\le
     (2c^3+c^2)(\sum\nl_{n<j\le  n+k}  \vp_j) +2c^3  \vp_{n+k+1} 
        , $$
    while replacing $n$ by $n+k+1$ 
    in \eqref{e10} yields
$$\|  T_{n+k+1}T_{n+k,n+1}T_n(x)  T_{n+k+1}  T_{n+k,n+1}T_n(y)  
-T_{n+k+1}  [T_{n+k,n+1}T_n(x)  T_{n+k,n+1}T_n(y)]          \|\le  \vp_{n+k+1} c^2    $$
and  hence  adding  these  last  two  inequalities  we  find
$$\| T_{n+k+1,n+1}T_n(x)   T_{n+k+1,n+1}T_n(y)-  T_{n+k+1,n+1}[T_n(x)T_n(y)]    \|$$
$$\le 
 (2c^3 +c^2)( \sum\nl_{n<j\le  n+k}  \vp_j
        +\vp_{n+k+1})   ,$$
   which completes the induction, whence proving of  the claim.
  From  the  claim  we  deduce
$$\|  ab-  Q\t_{n+1}  (T_n(x)T_n(y))\|  \le (2c^3 +c^2) \d_n    \|x\|\|y\|$$
and  \eqref{e15}  follows  by  \eqref{e11}  and  \eqref{e12}.  
But  now  since  $Y_n\subset  Y_{n+k}$, \eqref{e15} also implies
\begin{equation}  \label{ee5}\forall  a,b\in  Y_n\quad  d(ab,  Y_{n+k+1})\le    (2c^3+c^2)\d_{n+k}(1+\eta_{n+k})^2 \|a\|\|b\|  \to  0,\end{equation} 
and  hence  $ab\in  \ovl{\cup  Y_n}  =A$.
Clearly  the  same  conclusion  hods  for  any  $a,b\in  {\cup  Y_n}$,  so  that  $A$
(which,    as  we  already  noticed,  is  s.a.)      is  a  $C^*$-subalgebra
of  ${\cl L}$.

  We  will  now  show  that  $A$  has  the  WEP.  By  Proposition  \ref{p1},
this  reduces  to    the  following:

  \begin{ass1}\label{a1}  Fix  $n$    and  let  $u:  S  \to  A$  with  $S\subset  \ell_1^n$  and  $\|u\|_{cb}  \le  1$.
  For  any  $\vp>0$    there  is  an  extension  of  $u$  denoted  by
  $\tilde  u:  \ell_1^n  \to  A$  such  that  
  $ \tilde  u_{|S} =u  $  and  $\|\tilde  u\|\le  1+\vp$.
  \end{ass1}
  To  check  Assertion  \ref{a1}    we  may  obviously
  assume  by  density  that  $u(S)\subset  \cup  Y_m  $,  or  equivalently
  that  $u(S)\subset  Y_m$  for  some  $m\ge  n$  that  can  be  chosen  as  large  as  we  wish.
  Note  that  we  have  a  natural  embedding  $\ell_1^n  \subset  \ell_1^m$,
  with  which  any  $S\subset  \ell_1^n$  can  be  viewed  without  loss  of  generality
  as  sitting  in  $  \ell_1^m$,    and  for  the  map
  $v=w_m^{-1}  u:  S  \to  E_m$    we  have    $\|v\|_{cb}=\|w_m^{-1}  u\|_{cb}  \le  1+\eta_m$.
Taking  this  last  remark  into  account,  
  by  (i)  
in  Lemma  \ref{l3}    applied  to  $E_m$,  after  restricting  the  resulting  map  to  $\ell_1^n$,  we  find
a  map  $\tilde  v  :  \ell_1^n  \to  E_{m+1}$  such  that
$  \tilde  v_{|S}  =  T_m  v $  and  $$\|  \tilde  v\|  \le  (1+\vp_m)    \|v\|_{cb}\le  (1+\vp_m)  \|w_m^{-1}\|_{cb}\|u\|_{cb}
\le  (1+\vp_m)  (1+\eta_m)  .$$
Let  $\tilde  u=  w_{m+1}    \tilde  v  :  \ell_1^n  \to  Y_{m+1}$.
Then  $    \tilde  u_{|S} = w_{m+1} T_m w_m^{-1} u  = u $
and  
$$\| \tilde  u\|
 \le  \|w_{m+1} \| (1+\vp_m)  \|w_m^{-1}\|_{cb}\|u\|_{cb} \le (1+\eta_{m+1})(1+\vp_m)(1+\eta_m) \|u\|_{cb}$$
Since  $m$  can  be  chosen  arbitrarily  large
and  both  $\vp_m\to  0$  and  $\eta_m\to  0$,  we  obtain  Assertion  \ref{a1}.
By  Proposition  \ref{p1},  $A$  has  the  WEP.

By Remark \ref{rus3}, for  any  f.d.  subspace  $E\subset  C$  of  a  $C^*$-algebra  $C$  with  LLP
we  have  $d_{S\C}(E)=1$.  This  holds  in  particular  for  $C=C_0(\C)$.
Since    $d_{cb}(Y_n,E_n)\le \|w_n\|_{cb}\|w_n^{-1}\|_{cb} \le (1+\eta_n)^2$  
and  $E_n\subset  C_0(\C)$  we  have  $d_{S\C}(Y_n)\le  (1+\eta_n)^2$  for  all  $n$.
Since  $A=\ovl{\cup  Y_n}  $  with  $(Y_n)$    increasing,  we  have
$d_{S\C}(Y_k)\le  d_{S\C}(Y_n)$  for  all  $k\le  n$  and  hence  $d_{S\C}(Y_k)=1$
for  all  $k$.  By  perturbation  this  implies
$$\forall  E\subset  A\quad  
d_{S\C}(E)=1.$$
Since  $A$  has  the  WEP,  Proposition  \ref{p2}  implies  that  $A$  also  has  the  LLP.

Lastly,  since  we  have  $Z_n\subset  E_n$  for  all  $n$
there  is    $Z'_n\subset  Y_n$    such  that  $d_{cb}(Z'_n,Z_n)\le  d_{cb}(Y_n,E_n)\le  (1+\eta_n)^2$.
This  is  not  quite  what  is  stated  in  Theorem  \ref{t1}.  But
if  we  arrange  the  sequence  $(Z_n)$  so  that  each  space  in  it  is  repeated  infinitely  many  times,
then  for  any  given  space  $Z  $  in  the  sequence  $\{Z_n\}$  there  will  be
$Z'_n\subset  Y_n$  satisfying  $d_{cb}(Z'_n,Z  )\le  (1+\eta_n)^2$
for    infinitely  many
$n$'s.  Choosing  $n$  large  enough  so  that  $ (1+\eta_n)^2< 1+   \vp$,  we  obtain  
the  second  part  of  Theorem  \ref{t1}.
  \end{proof}
  
    \section{Possible  variants}
    
  \begin{ass}\label{1} We can avoid the use of
   the separable $C^*-$subalgebra $B\subset B(H)$ in our construction: we use $B(H)$ instead,
     a quotient map $C^*(\F) \to B(H)$ (for some large enough free group $\F$) and the fact that any separable
     $C^*$-
     subalgebra of $C^*(\F)$ lies in a copy of $\C$ embedded in $C^*(\F)$.
     \end{ass}

\begin{ass}\label{2} Using perturbation arguments, one  could  work  with  subspaces  $E$  such  that  $E  \subset  {\cl P}_0\otimes  \C$  
    where ${\cl P}_0={\rm span}[t^m\mid m\ge 1] \subset C_0 $ is the space of polynomials.
\end{ass}

      \begin{ass}\label{4} We  can  actually  work  with  non  s.a.  subspaces,  and  impose  an  additional  condition  that
      $\{T_n(x)^*  \mid  x\in  E_n\}  \subset  E_{n+1}$  together  with
      $\|T_n(x^*)  -  T_n(x)^*\|\le  \vp_n  \|x\|$  for  any  $x\in  T_{n-1}(E_{n-1})$.
      We  then  will  be  able  to  conclude  just  the  same  that  $A$  is  s.a.
    \end{ass}

    \begin{ass}\label{5}  The  construction  
      works  just  as  well  if  we  use  all  subspaces  of  $\C$  instead  of
      $\{\ell_1^n\}$,  in  the  style  of  
            Remark  \ref{R1}  (with  $C=\C$).
            More  precisely,  let  $X_n\subset  \C$  be  an  increasing  family
            of  f.d.  subspaces  with  dense  union.  We  may  replace  
            $S\subset  \ell_1^n$
            by  $S\subset  X_n$,  and  again  study  the  extension  problem
            of  $u:  S  \to  E$  by  $\tilde  u  :  X_n  \to  E_1$.
            This  shows  that,  while  using  $\ell_1^n$  seems  simpler,  there  is  nothing  special  about  it,  except  for  the  duality  used  in  the  proof  of  Proposition  \ref{p1} and the fact that $\|v\|=\|v\|_{cb}$ for any $v$
            defined on it.
            \end{ass}
\begin{ass}\label{6}
   By  the  main  result  described  in  \cite{P7}
    the  following  property  of  a  $C^*$-algebra    $A$  is  sufficient  (and  necessary)
    for    the  WEP:\\
    {\it  For  any  $n  \ge  1$,  any  map  $u:  \ell_\infty  ^n  \to  A$
    with  $\|u\|_{cb}\le  1$  and  any  $\vp>0$,
    there  are  $a_j,b_j\in  A$  such  that  $\sup\nl_{1\le  j\le  n}\|u(e_j)-a_jb_j\|\le  \vp$
    and  $\|\sum  a_ja_j^*\|^{1/2}  \|\sum  b_j^*b_j\|^{1/2}  \le1+\vp$.}
    
    Indeed,  this  implies  that  $\|u\|_{dec}\le  1$  for  any  such  $u$.  One  can  use  this  criterion  instead  of  the  one  in  Proposition  \ref{p1}
    to  construct  our  main  example.
    \end{ass}
    \begin{ass}\label{7} In fact we can avoid the use of the preceding result
       using all the algebras $M_N(A)$ in place of $A$. Then we may restrict to $n=3$.
       Indeed, by the criterion in \cite{P}, the pair $(A,\C)$ is nuclear
       (i.e. $A$ has the WEP) if and only if  for any $N\ge 1$ the algebra $M_N(A)$ satisfies the factorization in the preceding point \ref{6}, restricted to $n=3$.
    \end{ass}

            \section{A more general viewpoint}
            
            Perhaps the most general
           way to describe the applicability of the preceding construction is
           as follows.
           We assume given an isometric $*$-homomorphism
            $i: \cl C\to B$ and a quotient $*$-homomorphism
           $q: \cl C \to B$ where $\cl C,B$ are separable $C^*$-algebras and $\cl C$ is assumed to have the 
           LLP.  We   assume that $q$ almost allows liftings as defined above in Definition
           \ref{rd1} (this   automatically holds when we pass to the cone algebras).    
           Suppose we are given a ``suitable" (as defined next) property  $\cl P$ of $*$-homomorphisms 
           between $C^*$-algebras. Then if 
             the inclusion
           $i: \cl C \to B$ satisfies $\cl P$, we can construct a separable $C^*$-algebra $A$
           with $d_{S\cl C}(A)=1$ (and hence $d_{S\C}(A)=1$ by Remark \ref{rus3}) so that the identity of $A$ satisfies
           that same property.
           Our goal in this section is to prove this  in an even more general 
           setting that we spell out in Theorem \ref{rt1}.

            Let $\cl P$ be a property of $*$-homomorphisms $\sigma: C \to B$
           between $C^*$-algebras.  We say that $\cl P$ is suitable
           if it is inherited by any 
            $\sigma_1: C_1 \to B_1$ satisfying for some constant $c$ the following
            local factorization through $\sigma$: 
            for any f.d.s.a. subspaces $Y\subset C_1$,   $\cl E^0 \subset B$  and $\vp>0$ there
            are f.d.s.a. subspaces $E\subset C$
            , $\cl E\subset B$ such that $\sigma(E)\subset \cl E$
            with $ \cl E\supset\cl E^0$
            together with 
            a map $\beta: Y  \to E$ with $\|\beta\|_{cb}\le c$
            and
            an $\vp$-morphism (in the sense of
            Definition \ref{rd2}) $\gamma: \cl E \to B_1$
              such that ${\sigma_1}_{|Y}: Y \to B_1$
            admits a factorization of the form
            $$Y {\buildrel \beta \over\longrightarrow} E {\buildrel {\sigma_{|E}} \over\longrightarrow} \cl E {\buildrel \gamma \over\longrightarrow} B_1.$$ 
            For instance if $D$ is another $C^*$-algebra, we may consider
           the property that $Id_{D} \otimes \sigma$
           extends to a  contraction from $D \otimes_{\min}  C$ 
           to  $D \otimes_{\max} B$. This is an example of
            suitable property. 
            The case $D=\C$ corresponds to the WEP.
            We give   details on this in  Corollary \ref{c1}.\\
            Another example of suitable property 
            appears in the context of the similarity length
            in the sense of \cite[p. 401]{P4};
            let $\|\cdot\|_{(d)}$ be the norm  on $M_n(B)$ appearing in \cite[p. 401]{P4}
            when $B$ is an arbitrary $C^*$algebra.
            We say that  $\sigma: C \to B$ has $\cl P_d$
            if there is a constant $K$ such that, for any $n$,
            any   $x\in M_n(B)$ satisfies
            $\|( Id_{M_n} \otimes \sigma)(x)\|_{(d)} \le K \|x\|_{M_n(C)}.$
            A $C^*$-algebra $B$ is called of length $d$
            if  its identity map satisfies $\cl P_d$.
            It is known  (see \cite[Cor. 6]{Pij}) that 
            this holds for $B(H)$ with $d=3$ (with   $K=1$).
            From this it is easy to check that there is a separable $C^*$-algebra $B$
            containing $\C$ such that $Id_B$ satisfies $\cl P_3$.
            Using the latter  we can find a $C^*$-algebra $A$
            satisfying the properties in Theorem \ref{t1}
            and additionally of length 3.

            \def\L{\cl L}
            \begin{rem}\label{rr4}[A general setup] To achieve the greatest generality we
            are led to consider the following situation.
            Let $C_n,B_n$ be $C^*$-algebras ($n\ge 0$).
            Assume given, for each $n\ge 0$, an isometric $*$-homomorphism $i_n: C_n \to B_n$ 
            and   a surjective $*$-homomorphism $q_n: C_{n+1} \to B_n$ that almost allows liftings.
            Let $L=\ell_\infty(\{C_n\}) $, $\cl I_0 =c_0(\{C_n\})$ 
            and $\L=L/ \cl I_0$. We assume given,  for each $n\ge 1$, a certain
            correspondence $E \mapsto ({\cl E}[n,E], \vp[n,E])$ associating to a f.d.s.a. subspace $E$ of 
            $C_n$      a   f.d.s.a. subspace ${\cl E}[n,E]$ of $B_n$
            and a positive number $\vp[n,E]>0$.
           We also give ourselves
            a sequence of  f.d.s.a. subspaces $E_n^0\subset C_n$.
            The condition 
           $E_n \supset E_n^0$ in the next statement
           should be interpreted as expressing that $E_n$ is ``arbitrarily" large.

            \end{rem}
            $$\xymatrix{&&C_{n+1}\ar@{>>}[d]^{ q_n } \\
& C_n
 \ar@{^{(}->}[r]^{ i_n } &B_n }$$
 In short we are considering a sequence of $C^*$-algebras $(C_n)$
 such that  $C_n$ is a subquotient of $C_{n+1}$  for each $n$,
 and we assume that the quotient maps   almost allow liftings.
 
            The goal of the next theorem is 
            to show that there exists a $C^*$-algebra $A$ that has the same
            asymptotic  ``local" properties
            as the sequence of maps  $i_n : C_n \to B_n$. 
            As  before in \S \ref{out},
            we construct our $A$ as an inductive limit of operator spaces.        

             \begin{thm}\label{rt1} In the   situation described in Remark \ref{rr4}, 
             given  $\d_n>0$ with $\d_n\to 0$,
             if all the $C^*$-algebras
             $C_n$ have the LLP, 
             there is 
             a $C^*$-subalgebra $A\subset \L$ and an increasing sequence of f.d.s.a.
            subspaces $Y_n \subset A$ with $\ov{\cup_n Y_n}=A$ such that for each $n\ge 0$
            there are f.d.s.a. subspaces $E_n \subset C_n$ and $\cl E_n \subset B_n$ such that             
            $$i_n(E_n)\subset \cl E_n \text{  and also  }
  E_n \supset E_n^0 \text{  and  }  \cl E_n \supset \cl E[n,E_n],  $$ for which 
            the inclusion $Y_n \to Y_{n+1}$ admits a factorization of the following form
            $$Y_n {\buildrel \beta_n\over\to}  E_n \to \cl E_n {\buildrel \gamma_n\over\to} Y_{n+1}$$
      where $\beta_n$ is a linear isomorphism satisfying
            $$\max\{ \|\beta_n\|_{cb},  \|\beta^{-1}_n\|_{cb} \} \le 1+\d_n,$$ 
            and where $\gamma_n$ and $\beta_n^{-1}$ are $\d_n$-morphisms into $\L$ while
            $E_n \to \cl E_n$ is the restriction of the embedding $i_n: C_n \to B_n$.\\
            Moreover, we can ensure that, for each $n$, $\gamma_n$ is
            an $\vp[n,E_n]$-morphism.
            \end{thm}
            \begin{proof}
            We will construct the sequence $(E_n, \cl E_n)$ by induction, starting 
            from $E_0= E_0^0, \cl E_0 = \cl E[0,E_0] $. 
            The  induction reasoning will additionally produce,  for each $n$, a
            number $\vp_n'>0$  and an $\vp'_n$-morphism 
            $\psi_n : \cl E_n \to   C_{n+1}$ such that $ \psi_n(\cl E_n) \subset E_{n+1}$
            (so  that we may view $\psi_n$ as taking values in $E_{n+1}$), and $T_n=\psi_n {i_n}_{| E_n}: E_n \to E_{n+1}$ will be such that
            $\max\{\|T_n\|_{cb}, \|{T^{-1}_n}_{|T_n(E_n)}\|_{cb}\} \le 1+\vp'_n$. 
            Moreover, we will ensure that $\psi_n( \cl E_n )\psi_n( \cl E_n ) + E^0_{n+1} \subset E_{n+1}$
            (so that a fortiori $T_n(  E_n )T_n(  E_n ) + E^0_{n+1} \subset E_{n+1}$).
            \\The number $\vp_n'$ can be defined as
            $$\vp'_n=c 2^{-n-1}\min\{1, \vp[j,E_j],\d_j \mid j\le n \}$$
            where $c>0$ is a numerical constant to be adjusted in the end. Note that
\begin{equation}\label{bad}
\vp'_n\le c 2^{-n-1}, 
            \ \ \sum\nl_n^\infty \vp'_k \le c 2^{-n} \vp[n,E_n]
            \text{  and  }
            \sum\nl_n^\infty \vp'_k \le c 2^{-n} \d_n
            .\end{equation}
          Assume $(E_n, \cl E_n)$ has been constructed as well as $(T_k,\psi_k,\vp'_k)$ for $k<n$.
            Let $\psi_n: \cl E_n \to C_{n+1}$ be the $\vp$-morphism (for an $\vp$  to be specified) given by Lemma \ref{l5} 
            and let
             \begin{equation}\label{tee6} E_{n+1} =      \psi_n( \cl E_n ) + \psi_n( \cl E_n )\psi_n( \cl E_n )+ E^0_{n+1} ,\end{equation} and
            $$\cl E_{n+1} = \cl E[n+1,E_{n+1}] +  i_{n+1}(E_{n+1}) + i_{n+1}(E_{n+1})i_{n+1}(E_{n+1}) .$$           
            Note $\psi_n( \cl E_n )  \subset E_{n+1}$.
            We define $T_n : E_n \to E_{n+1}$ by $T_n=  \psi_n {i_n}_{| E_n}$.
            By Lemma \ref{l5}
            we have 
            $\|T_n\|_{cb}\le 1+\vp$. 
            Since $\|q_n\psi_n -Id\|\le \vp$ we have also 
            $\|q_nT_n  -{i_n}_{|E_n}\| \le \vp $
            and hence (see e.g. \cite[p. 75]{P4}) $\|q_nT_n  -{i_n}_{|E_n}: E_n \to B_n\|_{cb} \le \vp \dim(E_n) $.
            Since $i_n$ is isometric this gives us  $\|T_n(x)\|\ge \|q_nT_n(x)\|\ge (1-\vp)\|x\|$,
            and hence  $\|{T^{-1}_n}_{|T_n(E_n)}\| \le (1-\vp)^{-1}$
            and similarly (assuming $\vp\dim(E_n)< 1$) $\|{T^{-1}_n}_{|T_n(E_n)}\|_{cb} \le (1-\vp\dim(E_n))^{-1}$.
             Thus,  choosing $0<\vp < \vp'_n$ small enough we can 
            obtain  $E_{n+1}, \cl E_{n+1}, \psi_n,T_n $
            with all the required properties,  to complete the induction reasoning.
            
            Note that   $T_n=\psi_n {i_n}_{|E_n}$  is, just like $\psi_n$, 
            an $\vp'_n$-morphism into $C_{n+1}$.
            
            Let $\theta_n: E_n \to \ell_\infty(\{C_n\})$ 
            be defined by 
            \begin{equation}\label{bad1}\theta_n(x)=(0,\cdots,0,x,T_nx,T_{n+1}T_nx, \cdots)
            \end{equation}
             with $x$ standing at the $n$-th place,
            and
            let $w_n: E_n \to \L$ be such that   
             \begin{equation}\label{bad2}w_n=Q \theta_n ,\end{equation} 
             where $Q:  \ell_\infty(\{C_n\}) \to \L$ is the quotient map.\\            
            Let ${\eta_n}$ be such that $1+{\eta_n}=\prod\nl_n^\infty (1+\vp'_k)$. Note ${\eta_n} \to 0$. 
            Since $\|T_k\|_{cb}\le 1+\vp'_k$ we have    clearly
            $$\|w_n\|_{cb}\le\|\theta_n\|_{cb}\le 1+{\eta_n}.$$
            Moreover, since $\|{T^{-1}_k}_{|T_k(E_k)}\|_{cb}\le 1+\vp'_k$, we have for any $x\in E_n$
            $$\|w_n(x)\|=\limsup\nl_k\|(T_{n+k}\cdots T_{n+1}T_n)(x) \|\ge    (1+{\eta_n})^{-1}\|x\| ,  $$
            and hence $\|{w^{-1}_n}_{|w_n(E_n)} \|\le 1+{\eta_n}$.
            A similar reasoning applied to $x\in M_N(E_n)$ with $N$ arbitrary yields
            $\|{w^{-1}_n}_{|w_n(E_n)} \|_{cb} \le 1+{\eta_n}.$
            We set $Y_n=w_n(E_n)\subset \L$.
          Note 
          \begin{equation}\label{ee6} 
          \forall x\in Y_n\quad 
          w_{n}(x) = w_{n+1} T_n(x).  
          \end{equation}  
          Therefore    $Y_n\subset Y_{n+1}$. 
          Moreover,
          $Y_n$ is a f.d.s.a. subspace of $\L$.            
          The proof  that $A=\ovl {\cup Y_n}$ is a $C^*$-subalgebra
          is entirely analogous to that of Theorem \ref{t1}
          so we skip it.
          \\ By \eqref{ee3} with  \eqref{bad1} and \eqref{bad2}, $\theta_n$, and hence also $w_n$, is a $\d'_n$-morphism for some $\d'_n\to 0$
          with $\d'_n  \approx \sum_n^\infty \vp'_k$. We set $\beta_n=w_n^{-1}: Y_n \to E_n$ and
          $\gamma_n= w_{n+1}\psi_n: \cl E_n \to Y_{n+1}$.
          Then by \eqref{ee3} again $\gamma_n$ is a $\d''_n$-morphism for some $\d''_n\to 0$  with $\d''_n  \approx \sum_n^\infty \vp'_k$.
        We have $\gamma_n i_n \beta_n =w_{n+1} T_n w_n^{-1}$ 
        and hence by \eqref{ee6}  $\gamma_n i_n \beta_n(x)=x$ for any $x\in Y_n$.
        Lastly $\d'_n$ and $\d''_n$ being dominated (up to constant) by
        $\sum_n^\infty \vp'_k$ as in \eqref{ee3}, it is clear that our initial choice of the constant $c$ in the definition of   $(\vp'_n)$ 
        can be  adjusted small enough (see \eqref{bad}) in order to have   
        $\max\{\d'_n,\d''_n, \vp'_n,{\eta_n}\} \le \d_n$ and also
        $\d''_n\le  \vp[n,E_n]$. The latter shows that
        $\gamma_n$ is a fortiori a $\vp[n,E_n]$-morphism.
        This completes the proof. 
              \end{proof}
              \begin{rem} If we drop the LLP assumption, we only obtain $\|w_n\|=\|\beta_n^{-1}\|\le 1+\d_n$.
              \end{rem}
              The following diagram summarizes the preceding proof.
              $$\xymatrix{ 
                &   C_n\ar@{^{(}->}[r]^{  \  \  i_n  \quad  }  &  B_n&  {C_{n+1}}\ar@{>>}[l]_{ q_n}  & & \\  
  Y_n  \ar[r]^{w_n^{-1}}
  & E_n\ar@/_/[rr]_{T_n}\ar@{^{(}->}[u]\ar@{^{(}->}[r]^{  \  \  {i_n}{|E_n}  \quad  }  &  {\cl E}_n\ar@{^{(}->}[u]\ar[r]^{\psi_n} & E_{n+1}\ar@{^{(}->}[u] \ar[r]^{w_{n+1}} &Y_{n+1} &}$$

\begin{rem} In the general situation described in Remark \ref{rr4},
let $\L(\{C_n\})= \ell_\infty(\{C_n\}) /c_0(\{C_n\}) $
and $\L(\{B_n\})= \ell_\infty(\{B_n\}) /c_0(\{B_n\}) $. With the notation of the preceding proof,
we will say (in the style of \cite{BK97}) that $A\subset \L(\{C_n\})$ is the inductive limit 
of the system $\{E_n, T_n\}$, and we denote
$A=A(\{E_n, T_n\})$.
We define  $T_n^b: \cl E_n \to \cl E_{n+1}$ by
$T_n^b(x)= i_{n+1} \psi_n(x)$ ($x\in \cl E_n$).
Then we have  
 a $C^*$-subalgebra
$A(\{\cl E_n, T^b_n\})\subset \L(\{B_n\})$  that is  similarly  the inductive limit 
of the system $\{\cl E_n, T^b_n\}$. 
Let $Q^b: \ell_\infty(\{B_n\})\to \L(\{B_n\})$,
  and $w_n^b: \cl E_n \to \ell_\infty(\{B_n\})$
  be the analogues of $Q$ and $w_n$
  for the system $\{\cl E_n, T^b_n\}$ 
Let $q^\sharp: \L(\{C_n\})\to \L(\{B_n\})$
and $i^\sharp: \L(\{C_n\})\to \L(\{B_n\})$
 be the morphisms 
associated respectively to $(q_n)$ and $(i_n)$.
Let $\sigma: A \to \L(\{C_n\})$ be the (inclusion) embedding.
\\ With this notation,   we have
$i^\sharp \sigma=q^\sharp \sigma$
and moreover  
$i^\sharp \sigma (A)= q^\sharp\sigma(A)=A(\{\cl E_n, T^b_n\}).$
The details are easy to check using  
$\|   q_{n+k} T_{n+k} - {i_{n+k}}_{|E_{n+k}} \| \le \vp_{n+k}$
which shows that $i^\sharp \sigma  = q^\sharp\sigma$ on $Y_n$, together with 
$i_{n+1}T_n-T_n^b i_n=0$ on $E_n$
and $T_{n+k }\cdots T_{n+1} T_n^b - T^b_{n+k }\cdots T^b_{n+1} T_n^b =0$ on $\cl E_n$.

The $C^*$-algebra $$\cl A=\{x\in \L(\{C_n\})\mid q^\sharp(x)=i^\sharp(x) \}=\{ Q( (x_n))\mid   (x_n)\in \ell_\infty(\{C_n\}),  q_n(x_{n+1})=i_n(x_n) \ \forall n \},$$ 
which contains 
our $A$ as a $C^*$-subalgebra,  should probably be viewed
as the joint inductive/projective limit  of the system of subquotients $(C_n)$.
$$\xymatrix{ 
 A \ar@{^{(}->}[r]& {\cl A} \ar@{^{(}->}[d]  \ar@{^{(}->}[r]  & \L(\{C_n\})\ar[d]^{q^\sharp} \\
 & \L(\{C_n\})\ar[r]^{i^\sharp} & \L(\{B_n\}) }$$
\end{rem}

            \begin{cor}\label{c1}
            In the preceding situation, 
            assume $(C_n,B_n,i_n,q_n)=  (C, B,i,q) $  for all $n \in \NN$.
            Assume $C$ separable.
            Then for any separable $C^*$-algebra $D$ such that
            $i\otimes Id_D: C \otimes_{\min} D \to B \otimes_{\max} D$ is continuous
            (and hence isometric) we can obtain a separable $C^*$-algebra $A$ such that $(A,D)$  is a nuclear pair, and moreover such that $A$   and $C$ are locally equivalent (see Definition \ref{d17}).
   \end{cor}
   
   \begin{proof} This follows from the theorem by Lemma \ref{ll5}. We may assume
   $D=\ovl{\cup F_n}$ for an increasing sequence of f.d.s.a. subspaces $F_n \subset D$,
   and similarly
   $C=\ovl{\cup {\cl F}_n}$ for an increasing sequence of f.d.s.a. subspaces $\cl F_n \subset C$.
     By  Lemma \ref{ll5} applied to $i(E) \otimes F_n$, 
     we can select a number  $\vp[n,E]>0$ and
     a f.d.s.a. subspace $\cl E[n,E]\subset B$
     containing $i(E)$  associated to $\d= 1/n$ according to Lemma \ref{ll5}.
     
         Let $x\in Y_n \otimes D_n$. 
     We then apply the theorem
     with $\{E_n^0\}=\{\cl F_n\}$ where we make sure that each $\cl F_n$ is repeated infinitely many times in the sequence $\{E_n^0\}$.
     Since $i: C \to B$ transforms min-norms to max-norms,
     we have
     $$\|(i\beta_n\otimes Id_D)(x)\|_{B \otimes_{\max} D}
     \le \|(\beta_n\otimes Id_D)(x)\|_{C \otimes_{\min} D}
     \le (1+\d_n) \|x\|_{A \otimes_{\min} D}.$$
     Since $\gamma_n$ is an $\vp[n,E_n]$-morphism and  $(\gamma_n\otimes Id_D)(i\beta_n\otimes Id_D)(x)=x$ we then obtain
     $$\forall x\in Y_n \otimes D_n\quad \|x\|_{A \otimes_{\max} D} \le (1+{1/n})\|(i\beta_n\otimes Id_D)(x)\|_{B \otimes_{\max} D}\le (1+{1/n})(1+\d_n) \|x\|_{A \otimes_{\min} D}.$$
     Since $n$ can be chosen arbitrarily large this implies that $(A,D)$  is a nuclear pair.\\
     To show that $C$ locally embeds in $A$ it suffices to show by perturbation
     that each $\cl F_n$  locally embeds in $A$. Then since the inclusions $\cl F_n= E_m^0\subset E_m\subset C$ are valid for infinitely many $m$'s and
     $E_m$ is completely $(1+\d_m)^2$-isomorphic to $Y_m$, 
     we conclude that $C$ locally embeds in $A$. Since  $Y_n$ is completely
     $(1+\d_n)$-isomorphic to $E_n\subset C$, we know that $A$ is  locally embeds in $C$.
            \end{proof}

      \begin{proof}[Alternate proof of Theorem \ref{t1}]
      Let $B$ be as before a separable $C^*$-algebra with WEP
      containing $\cl C$ as a $C^*$-subalgebra. Then the inclusion
      $i: \cl C \to B$ satisfies 
      $$\|i
      \otimes Id_{\cl C} : \cl C \otimes_{\min} \cl C
    \to B \otimes_{\max} \cl C \|=1$$
     simply because  the identity of $B$ satisfies this by definition of the WEP.
      We apply Corollary \ref{c1} with
      $C_n=C_0(\cl C)$ and $B_n=C_0(B)$ for all $n\ge 1$ with $D=\cl C$.
      Lemma \ref{ll7} shows that the corresponding $q_n$'s
      almost allow liftings. By Corollary \ref{c1}
      the resulting $A$ has the WEP and locally embeds in $C_0(\cl C)$ (which has the LLP)
      and hence in $\cl C$. It follows that $A$ has the LLP by Lemma
      \ref{l4}.   
      \end{proof}
      \begin{rem} For  $C^*$-algebras $A,C$ the property that  $A$
      locally embeds  in $C$ implies that $A$ embeds in an ultrapower of $C$. In general 
    the converse does not hold. However it does  if $A$ has the LLP.
    Thus it does not significantly weaken Theorem \ref{t1} if we just say that
    each of $A$ and $\cl C$ embeds in an ultrapower of the other.
      \end{rem}
      
      \n{\it Acknowledgement.} I thank the referee for his careful reading.


\begin{thebibliography}{JP95}

  \bibitem{[Ar4]}  W.  Arveson,  Notes  on  extensions  of  $C^*$-algebras,  \emph{Duke  Math.  J.  }      {\bf  44}  (1977),      329--355.
    
\bibitem{Bla}   B. Blackadar, Operator algebras, \emph{Theory of $C^*$-algebras and von Neumann algebras},  Springer-Verlag, Berlin, 2006.

   \bibitem{BK97}   B. Blackadar and E.  Kirchberg,  Generalized inductive limits of finite-dimensional $C^*$-algebras. Math. Ann. 307 (1997),  343--380.

   \bibitem{BoP} J. 
Bourgain and G.  Pisier, 
A construction of ${\cl L}^\infty$-spaces and related Banach spaces, 
Bol. Soc. Brasil. Mat. 14 (1983),  109--123. 

      \bibitem{BO}  N.P.  Brown  and  N.  Ozawa,  \emph{$\mathrm{C}^*$-algebras  and  finite-dimensional  approximations},  Graduate  Studies  in  Mathematics,  88,  American  Mathematical  Society,  Providence,  RI,  2008.  
                            
\bibitem{CoH}
A. Connes and N. Higson,  {D\'eformations, morphismes asymptotiques et K-th\'eorie bivariante}  [Deformations, asymptotic morphisms and bivariant K-theory] \emph{C. R. Acad. Sci. Paris S\'er. I Math.}  {\bf 311} (1990),  101--106.                             

 \bibitem{ER}  E.  Effros  and  Z.J.  Ruan,  {\it  Operator  Spaces}.
  Oxford  Univ.  Press,  Oxford,  2000.
                                
\bibitem{Ha}  
U.  Haagerup,  Injectivity  and  decomposition  of  completely
  bounded  maps  in  ``Operator  algebras  and  their  connection  with  Topology  and
  Ergodic  Theory''.  Springer  Lecture  Notes  in  Math.  1132
(1985),  170--222.

  \bibitem{JP}  M.  Junge  and  G.  Pisier,    Bilinear  forms  on  exact  operator  spaces  and  $B(H)\otimes  B(H)$,  \emph{Geom.  Funct.  Anal.}  {\bf  5}  (1995),    329--363.
  
  \bibitem{Kir}  E.    Kirchberg,    On  nonsemisplit  extensions,  tensor  products  and  exactness  of  group  $C^*$-algebras,
\emph{Invent.  Math.}  {\bf  112}  (1993),  449--489.


  \bibitem{Kiuhf}  E.    Kirchberg,  Commutants  of
unitaries  in  UHF  algebras  and  functorial  properties  of
exactness,  \emph{J.
  reine  angew.  Math.}  {\bf    452}  (1994),
39--77.  

  \bibitem{Lan}  C.  Lance,
          On  nuclear  $C^*$-algebras.  J.  Funct.  Anal.  12  (1973),  157--176.



\bibitem{Lor}  T.  Loring,  \emph{Lifting  solutions  to  perturbing  problems  in  $C^*$-algebras},
Fields  Institute  Monographs,  Amer.  Math.  Soc.,  Providence,  RI,  1997.

  \bibitem{Oz1}  N.  Ozawa,  About  the  QWEP  conjecture,
  \emph{Internat.  J.  Math.    }  {\bf  15}  (2004),  501--530.



\bibitem{Oz2}  N.  Ozawa,    About  the  Connes  embedding  conjecture:  algebraic  approaches,
              \emph{Jpn.  J.  Math.}  8  (2013),  no.  1,  147--183.
              
 \bibitem{P0}  G.    Pisier,    Counterexamples    to  a  conjecture  of
Grothendieck  \emph{Acta  Math.}  {\bf  151},  (1983),  181--209.
  

    \bibitem{P}
G.    Pisier,  A  simple  proof  of  a  theorem  of  Kirchberg  and  related  results  on  $C^*$-norms,  \emph{J.  Operator  Theory}  {\bf  35}  (1996),    317--335.


\bibitem{Pij}  G. Pisier, Remarks on the similarity degree of
an  operator algebra, \emph{International J. Math.} {\bf 12}
(2001), 403--414.

\bibitem{P4}    
G.  Pisier,    \emph{Introduction  to  operator  space  theory},    Cambridge  University  Press,  Cambridge,  2003.  


\bibitem{P5}  G.  Pisier,  Remarks  on  $B(H)\otimes  B(H)$,
  \emph{Proc.  Indian  Acad.  Sci.  (Math.  Sci.)}      {\bf  116}      (2006),  423--428.
        
\bibitem{P7}  G.  Pisier,  On  a  Characterization  of  the  Weak  Expectation  Property  (WEP),  arxiv,  July  2019.

\bibitem{P6}  G.  Pisier,    \emph{Tensor  products  of  $C^*$-algebras  and  operator  spaces,  
  The  Connes-Kirchberg  problem},    Cambridge  University  Press,  to  appear.  \\
Available  at:    \url{https://www.math.tamu.edu/~pisier/TPCOS.pdf}

 \bibitem{LauR} 
 M. R\o rdam, F. Larsen and N. Laustsen,  
\emph{An introduction to K-theory for $C^*$-algebras},  
  Cambridge University Press, Cambridge, 2000.
  
   \bibitem{RSt} 
 M. R\o rdam and E.
 St\o rmer,  \emph{Classification of nuclear $C^*$-algebras. Entropy in operator algebras}, Encyclopaedia of Mathematical Sciences, 126.  Springer-Verlag, Berlin, 2002. 
   
 

  \bibitem{[Ta2]} M.  Takesaki, ,
 On the crossnorm of the direct product of C*-algebras,  \emph{  Tohoku
 Math. J} {\bf 16} (1964), 111--122.


\bibitem{Tak}  M.  Takesaki,  Theory  of  Operator  algebras,  vol.  I.  Springer-Verlag,  Berlin,  Heidelberg,  New  York,  1979.

  \bibitem{Tak2}  M.  Takesaki,  Theory  of  Operator  algebras,  vol.  II-III.  Springer-Verlag,  Berlin,  Heidelberg,  New  York,  2003. 
  
\end{thebibliography}
\end{document}